\documentclass[a4paper, 12pt]{article}
\pdfoutput=1

\usepackage[T1]{fontenc}
\usepackage[utf8]{inputenc}
\usepackage{authblk}

\usepackage{amsfonts,amsmath,amssymb,mathtools,dsfont,microtype} 
\usepackage{xcolor}
\usepackage[noBBpl]{mathpazo} 

\DeclareFontFamily{U} {cmr}{}

\DeclareFontShape{U}{cmr}{m}{n}{
	<-6> cmr5
	<6-7> cmr6
	<7-8> cmr7
	<8-9> cmr8
	<9-10> cmr9
	<10-12> cmr10
	<12-> cmr12}{}

\DeclareSymbolFont{Xcmr} {U} {cmr}{m}{n}

\usepackage[labelsep = period, labelfont = bf]{caption}
\usepackage{float,graphicx,subcaption}
\usepackage{booktabs,makecell}
\usepackage{enumitem}

\usepackage{tkz-euclide}
\tkzSetUpPoint[fill=black, size = 3pt]
\tkzSetUpLine[color=black, line width=0.6pt]
\tkzSetUpCircle[color=black, line width=0.6pt]
\usetikzlibrary{shapes.symbols}
\usepackage[left = 2.5cm, right = 2.5cm, top = 2.5cm, bottom = 2.5cm, headsep = 12pt, headheight = 15pt]{geometry}

\definecolor{lightseagreen}{rgb}{0.13, 0.7, 0.67}

\definecolor{secondcolor}{HTML}{5BFF00}
\newcommand{\secondopacity}{.3}
\definecolor{bordeaux}{RGB}{100,0,50}
\definecolor{darkblue}{RGB}{25, 25, 112}

\newcommand{\say}[1]{``#1''} 

\usepackage[hyphens]{url} 
\usepackage[linktoc = all, hidelinks, colorlinks, unicode=true, pagebackref=true]{hyperref} 
\usepackage[capitalise, compress, nameinlink, noabbrev]{cleveref} 
\hypersetup{linkcolor={{blue!70!black}}, citecolor={black}, urlcolor={blue!70!black}}
\hypersetup{linkcolor=bordeaux, citecolor = darkblue, urlcolor = darkblue}

\usepackage{amsthm,thmtools,thm-restate}
\declaretheorem[name = Theorem, numberwithin = section, style = plain]{theorem}

\declaretheorem[name = Conjecture, numberlike = theorem, style = plain]{conjecture}

\declaretheorem[name = Lemma, numberlike = theorem, style = plain]{lemma}
\declaretheorem[name = Observation, numberlike = theorem, style = plain]{observation}

\declaretheorem[name = Problem, numberlike = theorem, style = plain]{problem}
\declaretheorem[name = Remark, numberlike = theorem, style = definition]{remark}
\crefname{observation}{Observation}{Observations}
\crefname{conjecture}{Conjecture}{Conjectures}
\crefname{claim}{Claim}{Claims}
\crefname{problem}{Problem}{Problems}

\DeclareFontFamily{U}{matha}{\hyphenchar\font45}
\DeclareFontShape{U}{matha}{m}{n}{
	<5> <6> <7> <8> <9> <10> gen * matha
	<10.95> matha10 <12> <14.4> <17.28> <20.74> <24.88> matha12
}{}
\DeclareSymbolFont{matha}{U}{matha}{m}{n}

\DeclareMathSymbol{\specialuparrow}{\mathrel}{matha}{"D2}
\DeclareMathSymbol{\specialrightarrow}{\mathrel}{matha}{"D1}

\renewcommand*{\backref}[1]{}
\renewcommand*{\backrefalt}[4]{
	\ifcase #1 Not cited.%
	\or $\specialuparrow$#2%
	\else $\specialuparrow$#2%
	\fi%
}

\renewcommand{\epsilon}{\varepsilon}
\renewcommand{\ge}{\geqslant}
\renewcommand{\le}{\leqslant}

\renewcommand{\leq}{\leqslant}
\renewcommand{\emptyset}{\varnothing}

\DeclarePairedDelimiter{\abs}{\lvert}{\rvert}

\DeclareMathOperator{\stablesetpolytope}{SSP}
\newcommand{\ssp}[1]{\stablesetpolytope(#1)}
\DeclareMathOperator{\tperfectpolytope}{TSTAB}
\newcommand{\tstab}[1]{\tperfectpolytope(#1)}
\DeclareMathOperator{\hperfectpolytope}{HSTAB}
\newcommand{\hstab}[1]{\hperfectpolytope(#1)}
\DeclareMathOperator{\qperfectpolytope}{QSTAB}
\newcommand{\qstab}[1]{\qperfectpolytope(#1)}
\newcommand{\vasek}{Chv\'{a}tal }

\title{Colouring t-perfect graphs}

\date{}

\begin{document}
	
	\author{Maria Chudnovsky\thanks {Supported by NSF Grant DMS-2348219 and by AFOSR grant FA9550-22-1-0083.}}
	\affil{\small Department of Mathematics, Princeton University, Princeton, USA}
	\author{Linda Cook\thanks{Supported by the Gravitation programme NETWORKS (NWO grant no. 024.002.003) of the Dutch Ministry of Education, Culture and Science (OCW) and a Marie Skłodowska-Curie Action of the European Commission (COFUND grant no. 945045) }}
	\affil{\small Korteweg-de Vries Institute for Mathematics, University of Amsterdam, The~Netherlands}
	\author{James Davies}
	\affil{\small Trinity Hall, University of Cambridge, United Kingdom}
	\author{Sang-il Oum\thanks{Supported by the Institute for Basic Science (IBS-R029-C1).}}
	\affil{\small Discrete Mathematics Group, Institute for Basic Science (IBS), Daejeon, South Korea}
	\author{Jane Tan}
	\affil{\small All Souls College, University of Oxford, United Kingdom}
	\affil[ ]{Email: \textsf{\href{mailto:: mchudnov@math.princeton.edu}{mchudnov@math.princeton.edu}},
		\textsf{\href{l.j.cook@uva.nl}{l.j.cook@uva.nl}}
		\textsf{\href{jgd37@cam.ac.uk}{jgd37@cam.ac.uk}},
		\textsf{\href{mailto:sangil@ibs.re.kr}{sangil@ibs.re.kr}},
		\textsf{\href{mailto:jane.tan@maths.ox.ac.uk}{jane.tan@maths.ox.ac.uk}}}

	\maketitle
	
	\begin{abstract}
		Perfect graphs can be described as the graphs whose stable set polytopes are defined by their non-negativity and clique inequalities (including edge inequalities).
		In 1975, Chv\'{a}tal defined an analogous class of t-perfect graphs, which are the graphs whose stable set polytopes are defined by their non-negativity, edge inequalities, and odd circuit inequalities.
		We show that 
		t-perfect graphs are $199053$-colourable.
		This is the first finite bound on the chromatic number of t-perfect graphs and answers a question of Shepherd from 1995.
		Our proof also shows that every h-perfect graph with clique number $\omega$ is $(\omega + 199050)$-colourable.
	\end{abstract}

	\renewcommand{\thefootnote}{\fnsymbol{footnote}} 

	\renewcommand{\thefootnote}{\arabic{footnote}} 

	\section{Introduction}\label{sec:intro}
	
	Let $G=(V,E)$ be a graph. A \emph{stable} set of $G$ is a set of pairwise non-adjacent vertices. 
	We write $\chi(G)$ for the \emph{chromatic number} of $G$, that is, the minimum number of colours needed to colour the vertices of $G$ so that no two adjacent vertices have the same colour. Equivalently, $\chi(G)$ is the minimum number of stable sets needed to cover the vertex set of $G$.
	A \emph{clique} is a set of vertices that are pairwise adjacent.
	We write $\omega(G)$ for the size of the largest clique in a graph $G$, called the \emph{clique number} of $G$.
	A graph is \emph{triangle-free} if it does not have a clique of size three.
	
	The stable set problem is the problem of finding the maximum size of a stable set in a graph. Lov\'asz~\cite{Lovasz1994} described it as one of the simplest and most fundamental problems concerning graphs. 
	Since the stable set problem is NP-hard in general, it is natural to restrict to input graphs on which the stable set problem can be solved efficiently.
	
	In particular, one approach is to consider the polytope generated by stable sets and use the techniques from linear programming.
	For a subset $S$ of $V$, we write $\chi^S\in \mathbb{R}^V$ to denote its \emph{incidence vector}, that is a $0$-$1$ vector such that $\chi^S(v)=1$ if $v\in S$ and $\chi^S(v)=0$ otherwise.
	The \emph{stable set polytope} of a graph $G=(V,E)$ is defined as the convex hull of the incidence vectors of the stable sets of $G$. We denote it by $\ssp{G}$.
	Since the stable set polytope is a convex hull of a set of points, it can be described by some set of linear inequalities.
	If we can efficiently identify, for a given point $x^*$ outside the polytope, a linear inequality certifying that $x^*$ is outside the polytope, then by using the ellipsoid method, one can solve the maximum weight stable set problem in polynomial time~\cite{GLS1988}. 
	This problem of identifying such a linear inequality for a polytope is called the \emph{separation problem}.
	So if the separation problem for the stable set polytope of a graph can be solved efficiently, then the stable set problem can be solved efficiently for the graph.

	Here are some easy inequalities that are satisfied by points $x\in \mathbb{R}^{V(G)}$ in every stable set polytope of a graph $G$: 
	\begin{enumerate}[label=(\alph*)]
		\item\label{eq:nonnegative} \emph{(Nonnegativity)} $x_v\ge 0$ for every vertex $v$ of~$G$.
		\item\label{eq:edge} \emph{(Edge inequality)} $x_u+x_v\le 1$ for every edge $uv$ of $G$.
		\item \label{eq:clique} \emph{(Clique inequality)} $\sum_{v\in K} x_v\le 1$ for every clique $K$ of $G$.
		\item\label{eq:oddcycle}\emph{(Odd cycle inequality)} $\sum_{v\in V(C)} x_v\le \frac{\abs{V(C)}-1}{2}$ for every odd cycle $C$ of $G$.
	\end{enumerate}

	Let $\qstab{G}$ be the set of all vectors $x\in \mathbb R^{V(G)}$ satisfying \ref{eq:nonnegative} and \ref{eq:clique}.
	Remarkably, important theorems by Lov\'asz~\cite{lovasz1972normal} and Fulkerson~\cite{perfectgraphpolytope-fulkerson} on perfect graphs, as stated in \vasek~\cite{chvatal-introduces-t-perfect-1975}, imply that 
	\[ \ssp{G}=\qstab{G}
	\text{ if and only if $G$ is perfect.}\]
	Perfect graphs were introduced in the 1960s by Berge in terms of chromatic numbers and clique numbers.
	A graph $H$ is an \emph{induced} subgraph of a graph $G$ if $H$ can be obtained from $G$ by deleting vertices and all incident edges adjacent to them. 
	A graph~$G$ is called \emph{perfect} if $\chi(H)=\omega(H)$ for every induced subgraph~$H$ of $G$.
	In what has since become known as the Strong Perfect Graph Theorem, Chudnovsky, Robertson, Seymour, and Thomas~\cite{CRST2006} proved a theorem characterising the list of forbidden induced subgraphs for the class of perfect graphs. 
	
	Motivated by perfect graphs, \vasek~\cite{chvatal-introduces-t-perfect-1975} initiated the study of t-perfect graphs.
	Let $\tstab{G}$ be the set of all vectors $x\in \mathbb R^{V(G)}$ satisfying \ref{eq:nonnegative}, \ref{eq:edge}, and \ref{eq:oddcycle}.
	We define that a graph is \emph{t-perfect}\footnote{Here `t’ stands for `trou’, the French word for `hole'.}
	if $\ssp{G}=\tstab{G}$.

	\begin{figure}
		\begin{center}
			\begin{tikzpicture}[scale=1, every node/.style={circle, draw, fill=black, inner sep=1pt, minimum size=2pt}]
				\foreach \i in {0,1,2}
				{
					\node (v\i) at (0,\i) {};
					\node (u\i) at (2,\i) {};
					\node (w\i) at (3,\i+1){};
					\draw (v\i)--(u\i)--(w\i)--(v\i);
				}
				\draw (v1)--(u2)--(v0);
				\draw (v1)--(u0)--(v2);
				\draw (v0)--(u1)--(v2);
				\draw (w0)--(w1)--(w2);
				\draw (w0) [bend right] to (w2);
			\end{tikzpicture}\quad
			\begin{tikzpicture}[scale=1, every node/.style={circle, draw, fill=black, inner sep=1pt, minimum size=2pt}]
				\foreach \i in {0,1,2,3,4}
				{
					\node (v\i) at (90+72*\i:.6) {};
					\node (w\i) at (90+72*\i:1.5) {};
					\draw (v\i)--(w\i);
				}
				\draw (v0)--(v2)--(v4)--(v1)--(v3)--(v0);
				\draw (v0)--(w1)--(v2)--(w3)--(v4)--(w0)--(v1)--(w2)--(v3)--(w4)--(v0);
			\end{tikzpicture}
		\end{center}
		\caption{The only known $4$-critical t-perfect graphs in the literature. 
			On the left is the complement of the line graph of the complement of $C_6$ found by Laurent and Seymour~\cite[p. 1207]{schrijver2003combinatorial},
			and on the right is the complement of the line graph of the $5$-wheel found by Benchetrit~\cite{benchetrit-PhD,benchetrit20164critical}.}
		\label{fig:laurent-seymour}
	\end{figure}
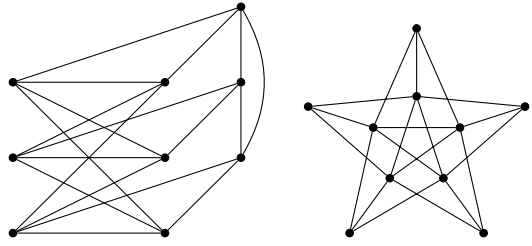
	Let us write $\mathbf{1}\in \mathbb{R}^{V(G)}$ for the vector with all entries equal to $1$.
	It is easy to verify that $\frac{1}{3}\mathbf{1}\in \tstab{G}$ for every graph $G$ and this implies that
	$K_4$ is not t-perfect and  
	the fractional chromatic number of any t-perfect graph is at most three.
	In 1992, Shepherd~\cite[8.14]{jensen-toft} asked whether for each t-perfect graph $G$, the polytope $\ssp{G}$ has the \emph{integer decomposition property}, that is, for every positive integer $k$, every integral vector in $k \ssp{G}$ can be written as a sum of $k$ vertices of $\ssp{G}$.
	If $G$ is t-perfect and $\ssp{G}$ has the integer decomposition property, then $\frac13\mathbf{1}\in \tstab{G}=\ssp{G}$ should be expressible as a sum of three incidence vectors of stable sets, implying that $G$ is $3$-colourable. 
	However, two counterexamples, depicted in~\cref{fig:laurent-seymour}, were found by Laurent and Seymour~\cite[p. 1207]{schrijver2003combinatorial} in 1994 
	and Benchetrit~\cite{benchetrit-PhD,benchetrit20164critical} in 2015, respectively, 
	answering Shepherd's question in the negative. On the other hand, Seb\H{o} conjectured that triangle-free t-perfect graphs are $3$-colourable~(see \cite{clawfreetperfect-maya}), and this is wide open.

	More generally, is every t-perfect graph $4$-colourable? This very natural question appears in the problem book of Jensen and Toft~\cite[8.14]{jensen-toft}, and is attributed to Shepherd from 1994. 
	Reiterating this, Shepherd wrote in the conclusion of his 1995 paper \cite{shepherd-lehman}:
	\begin{quote}
		\itshape
		For every $k\ge 4$, it is not known whether each t-perfect graph is $k$-colourable.
	\end{quote}

	Our main result is the first positive answer to this question.
	We remark that we have optimised the proof for simplicity rather than to optimise the bound.
	
	\begin{theorem}\label{thm:main}\mbox{}
		Every t-perfect graph is $199053$-colourable.
	\end{theorem}

	Let $\hstab{G}$ be the set of all vectors $x\in \mathbb R^{V(G)}$ satisfying \ref{eq:nonnegative}, \ref{eq:clique}, and \ref{eq:oddcycle}.
	A graph is \emph{h-perfect} if $\ssp{G}=\hstab{G}$.
	By their definitions, we have the relationships $\ssp{G}\subseteq \hstab{G}\subseteq \tstab{G}$ and $\hstab{G}\subseteq \qstab{G}$, and 
	t-perfect graphs are precisely h-perfect graphs without $K_4$ subgraphs. 
	Obviously, every perfect graph is h-perfect and every t-perfect graph is h-perfect.
	
	The study of h-perfect graphs was initiated by Sbihi and Uhry in 1984, who conjectured that every h-perfect graph $G$ with $\omega(G)\ge 3$ is $\omega(G)$-colourable~\cite[Conjecture 5.4]{SU1984}.
	This conjecture is false because of the graphs in \cref{fig:laurent-seymour}. 
	However, we prove that it is true up to an additive constant.  
	
	\begin{theorem}\label{thm:hperfect}
		Every h-perfect graph~$G$ is $(\omega(G)+ 199050)$-colourable.
	\end{theorem}
	
	We remark that \cref{thm:hperfect} is the first known result bounding the chromatic number of $h$-perfect graphs in terms of the clique number, establishing that the class of h-perfect graphs is a \say{$\chi$-bounded} class.
	For further background on t-perfect graphs, we refer to the chapter on t-perfect graphs in the book of Schrijver~\cite[Chapter 68]{schrijver2003combinatorial}.

	\subsection{Prior Work}
	
	Several subclasses of t-perfect graphs are known to be $3$-colourable.
	In 1982, Fonlupt and Uhry~\cite{fonlupt-uhry-1982} showed that 
	if a graph $G$ has a vertex whose deletion results in a bipartite graph, then it is t-perfect. Obviously, such a graph is $3$-colourable.
	
	The \emph{claw} is the bipartite graph $K_{1,3}$, $P_5$ is the $5$-vertex graph, and the \emph{fork} is the five vertex graph obtained from $K_{1,3}$ by subdividing one edge.
	We say that a graph $G$ is \emph{$H$-free} if $G$ has no induced subgraph isomorphic to~$H$.
	Bruhn and Stein~\cite{clawfreetperfect-maya}, Bruhn and Fuchs~\cite{P5free-BruhnFuchs}, and Cao and Wang~\cite{cao2022forkfree} showed respectively that the classes of claw-free t-perfect graphs, $P_5$-free t-perfect graphs, and fork-free t-perfect graphs are all $3$-colourable.
	
	Additionally, classes of graphs forbidding certain subdivisions of $K_4$ have been proven to be t-perfect and $3$-colourable.
	Series-parallel graphs are graphs that do not contain a $K_4$-subdivision as a subgraph and they are well known to be $3$-colourable because they always have a vertex of degree at most two.
	Series-parallel graphs are t-perfect, as conjectured by Chvat\'al~\cite{chvatal-introduces-t-perfect-1975} and proved by Boulala and Uhry~\cite{Boulala-Uhry-1979-series-parallele}.
	
	Subsequently, much stronger versions of this statement were also established.
	We call a graph an \emph{odd $K_4$} if it is a subdivision of $K_4$ such that every triangle of the $K_4$ maps to an odd cycle.
	In 1986, Gerards and Schrijver~\cite{GerardsSchrijver-NoOddK4istperfect-1986} showed that graphs containing no odd $K_4$ as a subgraph are t-perfect.
	Catlin~\cite{Catlin1979} showed that graphs having no odd $K_4$ as a subgraph are $3$-colourable.

	For further strengthening, we say that a \emph{bad $K_4$} is a subdivision of $K_4$ that is not t-perfect.\footnote{Every bad $K_4$ is an odd $K_4$, but the converse is false. Barahona and Mahjoub prove a characterisation of bad $K_4$'s \cite{BM1994b, schrijver2003combinatorial}.
	}
	In 1998, Gerards and Shepherd~\cite{GS1998} showed that graphs having no bad $K_4$ as a subgraph are t-perfect. In addition, they showed that graphs containing no bad $K_4$ as a subgraph are $3$-colourable\footnote{In \cite{GS1998}, Gerards and Shepherd stated this but omitted the proof. However, it can be easily deduced from their structural description.},
	extending the result of Catlin~\cite{Catlin1979}.

	\paragraph{Connection with $\chi$-boundedness}
	Motivated by the theory of perfect graphs, Gy\'arf\'as~\cite{Gyarfas1987} introduced the concept of $\chi$-boundedness. 
	A class $\mathcal{C}$ of graphs is said to be \emph{$\chi$-bounded} if there exists a function $f$ such that every induced subgraph $H$ of a graph $G \in \mathcal{C}$ satisfies $\chi(H) \leq f(\omega(H))$. 
	Such a function~$f$ is called a \emph{$\chi$-bounding function} for $\mathcal{C}$.
	If a class of graphs has a polynomial $\chi$-bounding function, then it is called \emph{polynomially $\chi$-bounded}.
	If it has a linear $\chi$-bounding function, then it is called \emph{linearly $\chi$-bounded}.
	Trivially, the class of perfect graphs is linearly $\chi$-bounded.
	In general, hereditary $\chi$-bounded classes of graphs can have arbitrarily fast-growing $\chi$-bounding functions \cite{brianski2024separating}.
	Classical constructions of Zykov, Mycielski, and Tutte show that the class of all graphs is not $\chi$-bounded \cite{descartes1947three, descartes1954solution, Mycielski1955, zykov1949}.
	See a survey by Scott and Seymour~\cite{SS2018} for more on $\chi$-boundedness.
	
	\cref{thm:hperfect} implies that the class of h-perfect graphs is linearly $\chi$-bounded and is the first result showing are $\chi$-bounded by any function. 
	Since t-perfect graphs are precisely $K_4$-free h-perfect graphs, the statement that the class of h-perfect graphs is $\chi$-bounded implies that there is a constant $k$ so that every t-perfect graph is $k$-colourable.
	Because \cref{thm:hperfect} is a corollary of \cref{thm:main}, this is not an alternative proof of a finite bound on the chromatic number of $t$-perfect graphs.

	\subsection{Our approach}
	
	Let us now sketch how we prove \cref{thm:main}.
	Our first step is to reduce the problem into the case where there are no \say{short} odd cycles. 
	For that, we will use a nice lemma, which
	originated from Sbihi and Uhry~\cite{SU1984} and was explicitly stated in the Ph.~D.~thesis of Marcus~\cite{Marcus1996}.
	It allows us to deduce that we can intersect all shortest odd cycles in a t-perfect graph by one stable set.
	Removing such a stable set increases the length of the shortest odd cycle by at least two and decreases the chromatic number by at most one.
	By repeatedly applying this argument, we reduce our problem to proving a bound on the chromatic number of an arbitrary t-perfect graph $G$ with no odd short odd cycles. In our case, we only need that $G$ has no odd cycles of length less than $11$.
	
	After this reduction step, we will use techniques developed for proving $\chi$-bounded\-ness for graph classes.
	We assume for a contradiction that $G$ has \say{large} chromatic number and no short odd cycles.
	By a standard inductive argument, we may assume that $G$ is connected. 
	The first tool we use is to partition the vertex set into levels $L_0$, $L_1$, $L_2$, $\ldots$ where $L_i$ is the set of vertices of distance exactly $i$ from a fixed vertex $v$.
	By a standard and simple method for proving $\chi$-boundedness, there is an integer~$i$ so that $\chi(G[L_i])\ge \frac12 \chi(G)$.
	Since there are no short odd cycles, $i$ has to be larger than some constant.
	
	Since $G[L_i]$ has large chromatic number, we can prove that it must contain various structures as induced subgraphs.
	We will use this fact to show that $G$ must contain an induced subgraph that is not $t$-perfect, a contradiction.
	An \emph{odd wheel} is a graph consisting of an odd cycle and a vertex adjacent to every other vertex on the cycle called its \emph{center}. It is an easy consequence of the definition of t-perfection that odd wheels are not t-perfect.
	Graphs with high chromatic number need not contain any wheels as an induced subgraph \cite{davies2023triangle,pournajafi2023burling}, but fortunately for us it is enough to show that $G$ contains an odd wheel as a \say{t-minor} \cite{GS1998}.
	A \emph{t-minor} of a graph~$G$ is a graph obtained from~$G$ by a sequence of vertex deletions and t-contractions, where a \emph{t-contraction} is an operation that is to contract all edges incident with a fixed vertex~$w$ when the neighbourhood of $w$ is a stable set.
	It is known that every t-minor of a t-perfect graph is t-perfect, 
	shown by Gerards and Shepherd~\cite{GS1998},

	The odd cycles of a graph are a critical in determining whether it is t-perfect. 
	By definition, the parity of cycles is preserved under taking t-minors and so we need to be careful with parity when we hunt for an odd wheel t-minor in $G$.
	We overcome this obstacle by introducing a structure we call an \say{arithmetic rope}, which we find to be of independent interest. By using the assumption that $G[L_i]$ has large chromatic number,
	we will show that $G[L_i]$ contains a 5-arithmetic rope.
	Our proof method is similar to that of Chudnovsky, Scott, Seymour, and Spirkl~\cite{CSSS20}, who showed that the class of graphs without long odd induced cycles is $\chi$-bounded.
	
	Informally, our $5$-arithmetic rope consists of $5$ vertices $q_1$, $q_2$, $q_3$, $q_4$, $q_5$ (with $q_6=q_1$) inside $L_i$ pairwise far apart in~$G$ such that for each $j\in\{1,2,\ldots,5\}$, $G[L_i]$ has both an odd-length induced path and an even-length induced path from $q_j$ to $q_{j+1}$
	so that any choice of one of the two paths for each $j$ gives an induced cycle containing all $5$ vertices.
	
	Once we have shown that $G[L_i]$ contains a 5-arithmetic rope $\mathcal{Q}$, to conclude that $G$ contains an odd wheel t-minor it suffices to show that there is a way of extending some choice of an odd cycle of our 5-arithmetic rope into an odd wheel t-minor. This will follow from the fact that $G$ has no short odd cycles, the definition of the levelling, and the property that the vertices $q_1, q_2, q_3, q_4, q_5$ are pairwise far apart. These arguments, along with the idea of arithmetic ropes, comprise the bulk of what is novel in our paper.
	We provide a more detailed sketch below.

	\paragraph{Extracting an odd wheel t-minor}

	For each~$j\in\{1,2,\ldots,5\}$, we choose one vertex $x_j$ in $L_{i-1}$ that is adjacent to~$q_j$
	and a vertex $y_j$ in $L_{i-2}$ that is adjacent to~$x_j$.
	As $q_1$, $q_2$, $q_3$, $q_4$, $q_5$ are pairwise far apart, 
	these vertices $x_1$, $x_2$, $x_3$, $x_4$, $x_5$, $y_1$, $y_2$, $y_3$, $y_4$, $y_5$ are all distinct and 
	each of $x_1$, $x_2$, $x_3$, $x_4$, $x_5$ has degree $1$ in 
	the subgraph of $G$ induced by these $10$ vertices.
	A simple lemma will show that $G[L_0\cup L_1\cup \cdots \cup L_{i-3}\cup \{x_1,x_2,x_3,x_4,x_5,y_1,y_2,y_3, y_4,y_5\}]$ contains a connected bipartite induced subgraph~$H$ containing $\{x_1,x_2,x_3,x_4,x_5\}$.
	Three out of these five vertices, say $x_a$, $x_b$, and $x_c$ with $a<b<c$, will be on the same side of the bipartite subgraph
	by the pigeonhole principle.
	We will delete the other two vertices in $\{x_1,x_2,x_3,x_4,x_5\}\setminus \{x_a,x_b,x_c\}$ from $H$ to obtain a connected bipartite induced subgraph~$H'$. 
	Now, we take three odd-length induced paths from $q_a$ to $q_b$, 
	from~$q_b$ to $q_c$, and from~$q_c$ to $q_a$ 
	to obtain an odd induced cycle $C$ in $G[L_i]$ such that $q_a$, $q_b$, and~$q_c$ split $C$ into three odd-length induced paths.
	
	We then consider the subgraph $G'$ of $G$ induced by $V(H')\cup V(C)$.
	We will then show that this structure will give us an odd wheel as a t-minor of~$G'$. 
	By applying t-contractions to every vertex of $H'$ on the side of the bipartite subgraph $H'$ not containing $\{x_a, x_b,x_c\}$, 
	we will obtain a t-minor $G''$ where all vertices of $H'$ are identified as a single vertex, say~$w$.
	So, $G''$ is a graph consisting of an odd cycle~$C$ with an extra vertex $w$ such that $w$ is adjacent to $q_a$, $q_b$, and~$q_c$ and possibly more. 
	Let $G'''$ be the graph obtained from $G''$ by repeatedly applying t-contractions to degree-$2$ vertices.
	Since $q_a$, $q_b$, and $q_c$ split~$C$ into odd-length paths and each $t$-contraction preserves the parity of the length of these paths in $C$, 
	we deduce that 
	$G'''$ is an odd wheel with at least three vertices on the rim.
	This implies that an odd wheel is a t-minor of $G''$, contradicting the assumption that $G$ is t-perfect.
	We remark that this is the reason why we want to construct a $5$-arithmetic rope: if we do not enforce any parity conditions on the length of subpaths of~$C$ split by neighbours of~$w$, then we may end up identifying two of $q_a$, $q_b$, and $q_c$, giving a t-perfect~$G'''$, failing to provide a contradiction.

	\subsection{Organisation}
	This paper is organised as follows.
	In \cref{sec:preliminaries} we will review definitions and key properties t-perfect graphs, including t-minors.
	In \cref{sec:overview} we will give an overview of our proof of \cref{thm:main} and explain how we reduce the problem to graphs with no short odd cycles and how we deduce \cref{thm:hperfect} from \cref{thm:main}.
	In \cref{sec:ropes}, we will introduce the concept of arithmetic ropes, that will give us a cyclic structure with correct parity of the lengths of the paths inside a graph of large chromatic number.
	In \cref{sec:oddwheel}, we will complete the proof by finding an odd wheel as a t-minor in a graph of large chromatic number without short odd cycles.
	In \cref{sec:conclusion}, we will conclude the paper with discussions and open problems.

	\section{Preliminaries}\label{sec:preliminaries}
	\label{sec:prelim}
	
	All graphs in this paper are simple, having no loops and no parallel edges.
	For two graphs $H$ and $G$, we say $G$ is \emph{$H$-free} if $G$ has no induced subgraph isomorphic to $H$.
	A \emph{stable} set of a graph is a set of pairwise non-adjacent vertices.
	A graph is \emph{$k$-colourable} if its vertex set is a union of $k$ stable sets.
	The \emph{chromatic number} $\chi(G)$ of a graph $G$ is the minimum non-negative integer $k$ such that it is $k$-colourable.
	For $X\subseteq V(G)$, we often use $\chi(X)$ to denote $\chi(G[X])$, where $G[X]$ is the \emph{induced subgraph} of $G$ on vertex set $X$.
	A graph is called \emph{$k$-critical} if its chromatic number is $k$ and removing any vertex makes decreases the chromatic number. 
	A \emph{hole} is an induced cycle of length at least $4$. An \emph{anti-hole} is the complement of a hole.

	We write $N_G(v)$ to denote the set of all neighbours of a vertex $v$ in $G$
	and let $N_G[v]=\{v\}\cup N_G(v)$.
	For a non-negative integer $r$, we write $N_G^r(v)$ to denote the set of all vertices of distance exactly $r$ from $v$ in $G$
	and $N_G^r[v]$ to denote the set of all vertices of distance at most $r$ from $v$ in $G$.
	We may omit the subscripts $G$ if it is clear from the context.
	
	A vertex set $B$ \emph{covers} a vertex set $C$ if $B$ and $C$ are disjoint and 
	every vertex in $C$ has a neighbour in~$B$.
	For a graph $G$, the \emph{$G$-distance} between two vertices $x$ and $y$ is the length of a shortest path from $x$ to $y$.
	A vertex set $X$ is \emph{anticomplete} to a vertex set $Y$ if every vertex in $X$ is non-adjacent to all vertices in $Y$.
	
	A \emph{t-contraction} is an operation to contract all edges in $G[N_G[v]]$ for some vertex~$v$ whose set of neighbours is stable.
	A graph $H$ is a \emph{t-minor} of a graph $G$ if $H$ can be obtained from~$G$ by a sequence of vertex deletions and t-contractions. 
	Gerards and Shepherd~\cite{GS1998} showed the following.
	
	\begin{lemma}[Gerards and Shepherd~\cite{GS1998}]\label{lem:tminor}
		Every t-minor of a t-perfect graph is t-perfect.
	\end{lemma}
	For completeness, we include an expanded version of the proof in Gerards and Shepherd~\cite[(14)]{GS1998}.
	\begin{proof}
		It is an easy and well known fact that an induced subgraph of a t-perfect graph is t-perfect, see Eisenbrand et al.~\cite[Lemma 4.1]{combinatorial-t-perfect}.
		Thus, it is enough to prove that if $G$ is a $t$-perfect graph and $u \in V(G)$ has degree at least two and
		$N_G(u)$ is a stable set, then the graph $\tilde{G}$ obtained by contracting all edges incident is also $t$-perfect.  
		Let us write $\tilde u$ to denote the vertex of~$\tilde G$ corresponding to $u$ of $G$.
		
		Suppose that $x\in \tstab{\tilde G}$ is a vertex of $\tstab{\tilde G}$.
		We define $y\in \mathbb R^{V(G)}$ by 
		\[
		y_v:= \begin{cases}
			x_v & \text{if $v\notin N_G[u]$},\\
			x_{\tilde u} & \text{if $v$ is adjacent to $u$},\\
			1-x_{\tilde u} & \text{if $v=u$}.
		\end{cases}
		\] 
		Since $x$ is a vertex, there is a set of tight inequalities from the definition of $\tstab{\tilde G}$ that determines~$x$.
		Then it is easy to observe that $y$ is determined by the corresponding inequalities for $\tstab{G}$ as well as the additional constraints $x_u+x_{v}=1$ for every $v\in N_G(u)$.
		Thus, $y$ is also a vertex of $\tstab{G}$. 
		Since $G$ is t-perfect, $y$ is the incidence vector of some stable set $S$ of $G$.
		Thus, $S\cap N_G[u]$ is either $\{u\}$ or $N_G(u)$. 
		
		Let $\tilde S:=S\setminus \{u\}$ if $u\in S$ and $\tilde S:=(S\setminus N_G(u))\cup \{\tilde u\}$ otherwise.
		Then $\tilde S$ is a stable set of $\tilde G$ and $x$ is the incidence vector of $\tilde S$. Thus $x\in \ssp{\tilde G}$. Therefore $\tilde G$ is t-perfect.
	\end{proof}

	For positive integers $k\ge 3$, the \emph{$k$-wheel} $W_k$ is the graph consisting of a cycle of length~$k$ and a single additional vertex adjacent to every vertex of the cycle.
	An \emph{odd wheel} is any $k$-wheel with $k$ odd.
	We will need the simple fact that odd wheels are not t-perfect.
	\begin{lemma}[Folklore; see~{\cite[Proposition 3.6.5]{benchetrit-PhD}}]\label{lem:oddwheelforbid}
		No odd wheel is t-perfect.
	\end{lemma}
	\begin{proof}
		Let $G$ be an odd wheel with a vertex $v$ such that $G-v$ is a cycle.
		Clearly $\frac13\mathbf{1} \in \tstab{G}$.
		Suppose that $\tstab{G}=\ssp{G}$.
		Then $\frac13\mathbf{1} =\sum_{i=1}^k \lambda_i \mathbf{1}_{S_i}$ for some stable sets $S_1,\ldots,S_k$ and positive reals $\lambda_1,\ldots,\lambda_k$ such that $\sum_{i=1}^k \lambda_i=1$.
		We may assume that $\lambda_1=\frac13$ and $S_1=\{v\}$ because 
		for every stable set $S$ of~$G$, $v\notin S$ or $S=\{v\}$. 
		This implies that $S_2$, $S_3$, $\ldots$, $S_k$ are stable sets in $G-v$, and therefore $\abs{S_i}< \frac{\abs{V(G)}-1}{2}$ for all $i\in \{2,3,\ldots,k\}$.
		Then, $\frac{\abs{V(G)}}{3}=\mathbf{1} \cdot \frac13\mathbf{1} = \sum_{i=1}^k \lambda_i \abs{S_i} < \frac{1}{3}+\frac23\frac{\abs{V(G)}-1}{2}$, which is a contradiction.
	\end{proof}
	
	Observe that Lemma 2.2 immediately implies that no $t$-perfect graph contains a clique of size four.
	
	The \emph{odd girth} of a graph is the length of a shortest odd cycle. (If it has no odd cycle, then it is $\infty$.) We shall freely use the following basic observation throughout the paper.
	\begin{observation}\label{lem:ball}
		Let $r$ be a non-negative integer.
		If the odd girth of a graph~$G$ is larger than $2r+1$, then $\chi(G[N_G^r[v]])\le 2$ for every vertex~$v$.
		\qed
	\end{observation}
	
	\section{Reducing to the case where there are no short odd cycles}\label{sec:overview}
	
	Our main technical result which implies \cref{thm:main} and in turn \cref{thm:hperfect} is as follows:
	\begin{theorem}\label{thm:technical-main}
		Every graph $G$ with odd girth at least $11$ that contains no odd wheel as a t-minor is $199049$-colourable.
	\end{theorem}

	\paragraph{Remarks on \cref{thm:technical-main}}
	With more technical arguments we are able to prove that triangle-free graphs containing no odd wheel as a t-minor have bounded chromatic number, but we do not know how to do this for graphs containing triangles.
	It was recently shown by Carbonero, Hompe, Moore, and Spirkl \cite{carbonero2023counterexample} that there are $K_4$-free graphs with arbitrarily large chromatic number that contain no induced triangle-free graph with chromatic number more than four.
	
	In general, it is not true that graphs without a t-minor~$H$ are $\chi$-bounded.
	This is because if a graph $G$ contains a graph $H$ as a t-minor, then $G$ also contains $H$ as an induced minor, and Pawlik, Kozik, Krawczyk, Laso{\'n}, Micek, Trotter, and Walczak~\cite{pawlik2014triangle} showed that there are graphs $H$ such that the class of graphs not containing $H$ as an induced minor is not $\chi$-bounded.
	
	There are also other similar notions of wheels, such as a graph consisting of an induced cycle of length at least four and an extra vertex with at least three neighbours on the cycle, whose exclusion as an induced subgraph does not imply $\chi$-boundedness \cite{davies2023triangle,pournajafi2023burling}.
	\\
	
	The remainder of this section is dedicated to deriving \cref{thm:main} and then \cref{thm:hperfect} assuming \cref{thm:technical-main}.
	We will use a lemma in the Ph.D.~thesis of Marcus~\cite[Proposition 3.6]{Marcus1996} to reduce the problem to the case where there are no short odd cycles.
	Such an argument also appeared earlier in Sbihi and Uhry~\cite[Proposition 5.3]{SU1984}.

	\begin{lemma}\label{color-reduction}
		If the class of t-perfect graphs with no odd cycle of length at most $2k+1$ is $c$-colourable, then the class of t-perfect graphs is $(c+k)$-colourable. 
	\end{lemma}

	We write $\chi^*(G)$ to denote the \emph{fractional chromatic number} of a graph $G$, which is defined as the minimum $\sum_{S\in\mathcal S}f(S)$ 
	for the set $\mathcal S$ of all stable sets of~$G$
	subject to the conditions that 
	\begin{enumerate}[label=(\roman*)]
		\item $\sum_{S\in \mathcal S: v\in S} f(S)\ge 1$ for every vertex $v$ of $G$, 
		\item $f(S)\ge 0$ for all $S\in \mathcal S$.
	\end{enumerate}
	The following lemma is 
	due to Sbihi and Uhry~\cite[Theorem 5.2]{SU1984}, see Schrijver~\cite[(68.81)]{schrijver2003combinatorial}.
	\begin{lemma}\label{lem:fractionalcolor}
		Let $G$ be a t-perfect graph. Let $\ell$ be a positive integer. If $G$ has no odd cycle of length less than $2\ell+1$, then $\chi^*(G)\le 2+\frac{1}{\ell}$. The equality holds if and only if $G$ has an odd cycle of length $2\ell+1$.
	\end{lemma}
	We remark that Bruhn and Fuchs~\cite{P5free-BruhnFuchs} conjectured that a graph $G$ is t-perfect if and only if $\chi^*(H)=2+\frac{2}{\operatorname{oddgirth}(H)-1}$ for every non-bipartite t-minor~$H$ of~$G$.
	The above lemma is the easy direction of this. 
	For completeness, we include its proof.
	
	\begin{proof}
		Let $G=(V,E)$ be a t-perfect graph with no odd cycle of length less than $2\ell$.
		Then $\frac{1}{2+(1/\ell)}\mathbf{1}$ is in $\tstab{G}$ because every odd cycle of $G$ has length at least $2\ell+1$.
		Since $G$ is t-perfect, $\frac{1}{2+({1}/{\ell})}\mathbf{1}$ is a convex combination of the incidence vectors of stable sets of~$G$. 
		This means that $\mathbf{1}$ is a non-negative linear combination of the incidence vectors of stable sets of~$G$
		where the sum of the coefficients is $2+\frac{1}{\ell}$ and therefore $\chi^*(G)\le 2+\frac{1}{\ell}$.
		
		Since the cycle of length $2\ell+1$ has the fractional chromatic number $2+\frac{1}{\ell}$, the equality holds if and only if $G$ has an odd cycle of length $2\ell+1$.
	\end{proof}
	
	In the next lemma, we show that there is a stable set hitting all shortest odd cycles of a t-perfect graph, because the fractional chromatic number of a t-perfect graph is tight on its shortest odd cycle. This idea was used by Sbihi and Uhry~\cite[Proposition 5.3]{SU1984} to show that if all odd induced cycles of a t-perfect graph have the same length, then it is $3$-colourable. 
	Later, Marcus~\cite[Proposition 3.30]{Marcus1996} used this idea repeatedly to show that the chromatic number of a h-perfect graph~$G$ with at most $t$ distinct lengths of odd induced cycles is at most $\omega(G)+t$.
	
	\begin{lemma}[Marcus~{\cite[Proposition 3.6]{Marcus1996}}]\label{lem:reduce}
		Let $G$ be a t-perfect graph of odd girth at least $2\ell+1$.
		Then $G$ contains a stable set~$S$ such that 
		$\abs{S}\ge \frac{\ell}{2\ell+1}\abs{V(G)}$ and 
		$G-S$ has odd girth at least $2\ell+3$.
	\end{lemma}
	It is possible to deduce from \cref{lem:reduce} that $\chi(G)\le O(\log \abs{V(G)})$ for  t-perfect graphs~$G$, as explained in~\cite[Proposition 3.36]{Marcus1996}.
	As the paper of Marcus~{\cite[Proposition 3.6]{Marcus1996}} is written in French, we include a proof of \cref{lem:reduce} for completeness.
	
	\begin{proof}
		We may assume that $G$ has an odd cycle of length $2\ell+1$.
		By~\cref{lem:fractionalcolor}, $\chi^*(G)=2+\frac{1}{\ell}$.
		So there is a list of stable sets $S_1,\ldots,S_k$ of $G$ 
		and a list of positive weights $w_1,w_2,\ldots,w_k$ such that $w_1+w_2+\cdots+w_k=2+\frac{1}{\ell}$ and for each vertex $v$ of $G$, $\sum_{i: v\in S_i}w_i\ge 1$.
		Then $\abs{V(G)}\le  \sum_{i=1}^k w_i \abs{S_i}\le (\sum_{i=1}^k w_i)\max_{j} \abs{S_j}$ and therefore there is $j$ such that $\abs{S_j}\ge \frac{\ell}{2\ell+1}\abs{V(G)}$. By symmetry, we may assume that $j=1$. 
		
		Let $H=G-S_1$.
		Then $H$ is a t-perfect graph such that $\chi^*(H)=2+\frac{1}{\ell}-w_1<2+\frac{1}{\ell}$ and therefore $H$ has no odd cycle of length $2\ell+1$.
		Thus $S_1$ is a stable set that intersects every odd cycle of length $2\ell+1$.
	\end{proof}

	\cref{color-reduction} is a direct consequence of \cref{lem:reduce}. Assuming \cref{thm:technical-main}, we can now prove \cref{thm:main}.
	
	\begin{proof}[Proof of \cref{thm:main} assuming \cref{thm:technical-main}.]
		By \cref{lem:tminor,lem:oddwheelforbid}, no t-perfect graph contains an odd wheel as a t-minor.
		By \cref{color-reduction,thm:technical-main}, it then follows that every t-perfect graph is 
		$199053$-colourable.
	\end{proof}
	
	The same proof allows us to deduce the following theorem from \cref{thm:technical-main}.
	\begin{theorem}\label{thm:trianglefree}
		Every triangle-free t-perfect graph is $199052$-colourable.
		\qed
	\end{theorem}

	\begin{remark}
		\cref{thm:technical-main} is strictly stronger than \cref{thm:main}.
		There are $t$-imperfect graphs, such as sufficiently large even M\"{o}bius ladders \cite{shepherd-lehman}, which have arbitrarily large odd girth and do not contain an odd wheel as a t-minor.
	\end{remark}

	Next, we show that \cref{thm:trianglefree} implies \cref{thm:hperfect}.
	To see why this is the case, we use the following well-known fact about h-perfect graphs~$G$, due to Seb\H{o} in~\cite[Lemma 26]{clawfreetperfect-maya}, also in~\cite[Fact 3.2]{Marcus1996}.
	
	\begin{lemma}[{Seb\H{o} in~\cite[Lemma 26]{clawfreetperfect-maya}}, {\cite[Fact 3.2]{Marcus1996}}]\label{lem:reduceomega}
		Let $G$ be a h-perfect graph with $\omega(G)\ge3$. 
		Then $G$ contains a stable set $S$ such that $\omega(G-S)<\omega(G)$.
	\end{lemma}
	This follows from the observation that $\frac{1}{\max(\omega(G),3)}\mathbf{1}\in \hstab{G}$, see 
	Sbihi and Uhry~\cite[Theorem 5.2]{SU1984} or 
	Bruhn and Stein~\cite[(14)]{clawfreetperfect-maya}:
	\[
	\chi^*(G)=\omega(G) \quad\text{if $G$ is h-perfect and $\omega(G)\ge 3$.}
	\] 
	Since every induced subgraph of a  h-perfect graph is h-perfect, 
	by the same argument of \cref{lem:reduce}, we have the following proof. 
	\begin{proof}[Proof of \cref{lem:reduceomega}]
		There is a stable set~$S$ of~$G$ such that $\chi^*(G-S)<\chi^*(G)$, as in the proof of \cref{lem:reduce}. 
		We may assume that $\omega(G-S)\ge 3$, because otherwise it is trivial.
		Since both $G$ and $G-S$ are h-perfect, we have $\chi^*(G-S)=\omega(G-S)<\chi^*(G)=\omega(G)$.
	\end{proof}
	\begin{proof}[Proof of \cref{thm:hperfect} assuming \cref{thm:trianglefree}]
		Let $G$ be a h-perfect graph. 
		We may assume that $\omega(G)\ge 2$. 
		By \cref{lem:reduceomega}, every h-perfect graph~$G$ admits $\omega(G)-2$ stable sets such that after deleting all these stable sets, we have a triangle-free h-perfect induced subgraph~$H$. 
		Such a graph $H$ is t-perfect.
		By \cref{thm:trianglefree}, 
		$H$ is $199052$-colourable.
		Therefore, $G$ is $(\omega(G)+ 199050)$-colourable, as desired.
	\end{proof}

	The proof of \cref{thm:technical-main} is split into two parts.
	In the next section we show how to find so-called arithmetic ropes as induced subgraphs within graphs of large chromatic number and odd girth at least $11$.
	Then, in Section~\ref{sec:oddwheel}, we shall construct odd wheel t-minors in graphs of large chromatic number by using arithmetic ropes as a useful gadget.

	\section{Arithmetic ropes}
	\label{sec:ropes}
	
	An $r$-arithmetic rope is a graph consisting of $2r$ paths $Q_{1,1}, Q_{1,2}, \ldots , Q_{r,1}, Q_{r,2}$ with ends contained in a vertex set $\{q_1,\ldots , q_r\}$ so that (taking indices modulo $r$)
	\begin{itemize}
		\item for every $1\le i \le r$, we have that $Q_{i,1}$ and $Q_{i,2}$ are $q_1q_2$-paths with odd and even length, respectively.
		\item for every $(h_1, \ldots , h_r) \in \{1,2\}^r$, the graph $Q_{1,h_1} \cup Q_{2,h_2} \cup \cdots \cup  Q_{r,h_r}$ is an induced cycle, and
		\item the vertices $q_1,\ldots , q_r$ are pairwise at $G$-distance at least $5$.
	\end{itemize}
	
	We denote such an arithmetic rope by $(q_i,Q_{i,1},Q_{i,2})_{i=1}^r$.
	An example of such a graph is shown in \cref{fig:5arope}.
	The key feature of an $r$-arithmetic rope is that by choosing suitable $(h_1, \ldots , h_r) \in \{1,2\}^r$, one can control the parity of paths between $q_i$ and $q_j$ in an induced cycle containing $q_1, \ldots , q_r$.

	\begin{figure}
		
		\begin{center}
			\begin{tikzpicture}[scale=1, every node/.style={circle, draw, fill=black, inner sep=1pt, minimum size=2pt}]
				\node [label=$q_1$] at (0,3) (q1) {};
				\node [label=$q_2$] at (2,3.2) (q2) {};
				\node [label=$q_3$] at (4,2) (q3) {};
				\node [label=$q_4$] at (3,0) (q4) {};
				\node [label=$q_5$] at (1,0) (q5) {};
				
				\draw[thick,dotted] (q1) .. controls (0.5,3.5) and (1.5,2.5) .. (q2);
				\draw[thick,dotted] (q2) .. controls (3,4) and (3.5,1.5) .. (q3);
				\draw[thick,dotted] (q3) .. controls (4.5,1) and (3.3,-0.3) .. (q4);
				\draw[thick,dotted] (q4) .. controls (2,1) and (2,-.5) .. (q5);
				\draw[thick,dotted] (q5) .. controls (-0.5,.5) and (1.5,2.5) .. (q1);
				\draw[thick,dashed] (q1) .. controls (0.8,3) and (1.2,4) .. (q2);
				\draw[thick,dashed] (q2) .. controls (2.8,2) and (3,3) .. (q3);
				\draw[thick,dashed] (q3) .. controls (3.5,.5) and (5,-0.3) .. (q4);
				\draw[thick,dashed] (q4) .. controls (2,0) and (2,.5) .. (q5);
				\draw[thick,dashed] (q5) .. controls (-2,1.5) and (3.5,2) .. (q1);
			\end{tikzpicture}
		\end{center}
		\caption{A $5$-arithmetic rope. Dotted lines represent odd-length paths and dashed lines represent even-length paths.}
		\label{fig:5arope}
	\end{figure}
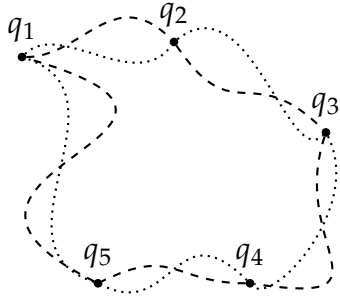
	
	The aim of this section is to prove the following.

	\begin{theorem}
		\label{thm:ropefinding}
		For any $r \in \mathbb{N}$, every graph $G$ with odd girth at least $11$ and $X\subseteq V(G)$ with $\chi(X) \ge 6^{r+1}+\frac{34}{5}(6^r-1)-1$ contains an $r$-arithmetic rope in $X$.
	\end{theorem}
	
	Chudnovsky, Scott, Seymour, and Spirkl~\cite{CSSS20} proved that the class of graphs without an induced long odd cycle is $\chi$-bounded.
	In one part of the proof, they find 
	three paths $P$, $Q_1$, $Q_2$ where $P$ is long, $Q_1$ has odd length, $Q_2$ has even length, and both $P\cup Q_1$ and $P\cup Q_2$ are (long) holes. By definition, one of $P \cup Q_1$, $P\cup Q_2$ is odd and the other is even. 
	Roughly, the idea of the proof of \cref{thm:ropefinding} is to carefully extend these ideas in order to iterate and find many such pairs of paths $Q_{i,1}$, $Q_{i,2}$ attached in order.
	Although doing this adds some difficulty, we are also able to simplify parts significantly due to the stronger odd girth  assumption.
	We remark that the idea from \cite{CSSS20} of finding paths $P, Q_1, Q_2$ is itself a variant of an older idea, \cite{recognizing-berge-graphs} introduces a precursor where the restriction on the length of $P$ is omitted.

	Let $G$ be a graph. We define a a stable grading to be an ordered partition of $V(G)$ into stable sets $(W_1, \dots, W_n)$.
	Formally, a \emph{stable grading} is a sequence $(W_1, \ldots , W_n)$ of pairwise disjoint stable sets of $V(G)$ such that their union is $V(G)$.
	We say that $u \in V(G)$ is \emph{earlier} than $v \in V(G)$ with respect to a stable grading $(W_1, \ldots , W_n)$ if $u \in W_i$ and $v \in W_j$ for some $i < j$.
	We require a couple of preliminary lemmas regarding stable gradings before we can start \say{building} part of an arithmetic rope.
	
	\begin{lemma}\label{lem:earlier1}
		Let $c \ge 1$ and let $(W_1, \ldots, W_n)$ be a stable grading of a graph $G$ with $\chi(G) \ge c + 2$. Then there exists a subset $X$ of $V(G)$ and an edge $uv$ of $G$ with the following properties:
		\begin{itemize}
			\item $G[X]$ is connected,
			\item $\chi(X) \ge c$,
			\item $u$ and $v$ are earlier than every vertex in $X$, and
			\item at least one of $u$ or $v$ has a neighbour in $X$.
		\end{itemize}
	\end{lemma}

	\begin{proof}
		We say that a vertex $w\in V(G)$ is \emph{left active} if there exists an edge $uv$ of $G$ such that $u$ and $v$ are earlier than $w$ and $w$ is adjacent to at least one of $u$ or $v$.
		Let $A$ be the vertices of $G$ that are left active, and let $B=V(G)\setminus A$.
		
		Suppose first that $\chi(A)\ge c$.
		Let $C$ be a connected component of $G[A]$ with $\chi(C)\ge \chi(A) \ge c$.
		Let $i$ be a minimum integer such that  $1 \le i \le n$ and $W_i \cap V(C)$ is non-empty and let $w$ be a vertex of $W_i \cap V(C)$. Since $w$ is left-active, there exists an edge $uv$ of $G$ such that $u$ and $v$ are earlier than $w$ (and thus every vertex of $C$), and $w$ is adjacent to at least one of $u$ or $v$. Setting $X=V(C)$, the lemma follows.
		
		So, we may assume that $\chi(A)\le c-1$, and thus $\chi(B)\ge 3$.
		Let $D$ be a connected component of $G[B]$ with $\chi(D)\ge \chi(B) \ge 3$.
		Let~$k$ be a minimum integer such that $1\le k \le n$ and $G\left[\left( \bigcup_{i=1}^k W_i \right) \cap V(D) \right]$ contains an odd cycle. Such an integer $k$ exists since $\chi(D)\ge 3$.
		Since each $W_i$ is a stable set, there exists some~$w\in W_i \cap V(D)$ and an edge~$uv$ of~$G$ with $u,v \in \left( \bigcup_{i=1}^{k-1} W_i \right) \cap V(D)$ such that $w$ is adjacent to at least one of $u$ or $v$.
		But then, $w$ is left active, contradicting the fact that $w\in B$.
	\end{proof}
	
	We extend this slightly in the triangle-free case (with a slightly worse bound) in order to say that only $v$ and not $u$ has a neighbour in $X$.
	
	\begin{lemma}\label{lem:earlier2}
		Let $c \ge 1$ and let $(W_1, \ldots, W_n)$ be a stable grading of a triangle-free graph $G$ with $\chi(G) \ge c + 3$. Then there exists a subset $X$ of $V(G)$ and an edge $uv$ of $G$ with the following properties:
		\begin{itemize}
			\item $G[X]$ is connected,
			\item $\chi(X) \ge c$,
			\item $u$ and $v$ are earlier than every vertex in $X$,
			\item $u$ has no neighbour in $X$, and
			\item $v$ has a neighbour in $X$.
		\end{itemize}
	\end{lemma}
	
	\begin{proof}
		By \cref{lem:earlier1}, there exists a subset $X'$ of $V(G)$ and an edge $u'v'$ of $G$ with the following properties:
		\begin{itemize}
			\item $G[X']$ is connected,
			\item $\chi(X') \ge c+1$,
			\item $u'$ and $v'$ are earlier than every vertex in $X'$, and
			\item $v'$ has a neighbour in $X'$.
		\end{itemize}
		Let $C$ be a connected component of $G[X'\setminus N(v')]$ with maximum chromatic number.
		Since $G$ is triangle-free, we have that $\chi(C)\ge \chi(X')-1\ge c$.
		If $u'$ has a neighbour in $V(C)$, then the lemma follows by setting $u=v'$, $v=u'$, and $X=V(C)$.
		So, we may assume now that $u'$ has no neighbour in $V(C)$.
		Let $w$ be a neighbour of $v'$ in $X'$ that has a neighbour in $V(C)$.
		Note that $w$ is non-adjacent to $u'$ since $G$ is triangle-free.
		Then, the lemma follows by setting $u=u'$, $v=v'$, and $X=V(C)\cup \{w\}$.
	\end{proof}

	The following lemma extracts the key induction step that we will use to prove \cref{thm:ropefinding}. 
	\begin{lemma}\label{lem:ropeinduction}
		Let $c \ge 1$ and let $G$ be a graph with odd girth at least $11$, let $B,C\subseteq V(G)$ be disjoint vertex subsets with $B$ covering $C$, let $q\in V(G)\setminus C$ be a vertex such that $G[C\cup \{q\}]$ is connected and $\chi(C) \ge 6c+17$.
		
		Then there exists $B'\subseteq B$, $C'\subseteq C$, $q'\in C\setminus C'$ and two induced paths $Q_0$, $Q_1$ between $q$ and $q'$ in $G[(C \setminus C') \cup (B \setminus B') \cup \{q\}]$ such that
		\begin{itemize}
			\item $G[C'\cup \{q'\}]$ is connected,
			\item $\chi(C') \ge c$,
			\item $B'$ covers $C'$,
			\item $B'$ is anticomplete to $(V(Q_0)\cup V(Q_1))\setminus N^{2}[q']$,
			\item $V(Q_0)\cup V(Q_1)$ is contained in $C \cup \{q\} \cup ((B\cap N^2(q'))\setminus N^3(q))$,
			\item $C'$ is anticomplete to $(V(Q_0) \cup V(Q_1))\setminus \{q'\}$,
			\item the $G$-distance between $q$ and $C'\cup \{q'\}$ is at least $5$, and
			\item $Q_0$ has even length and $Q_1$ has odd length.
		\end{itemize}
	\end{lemma}
	
	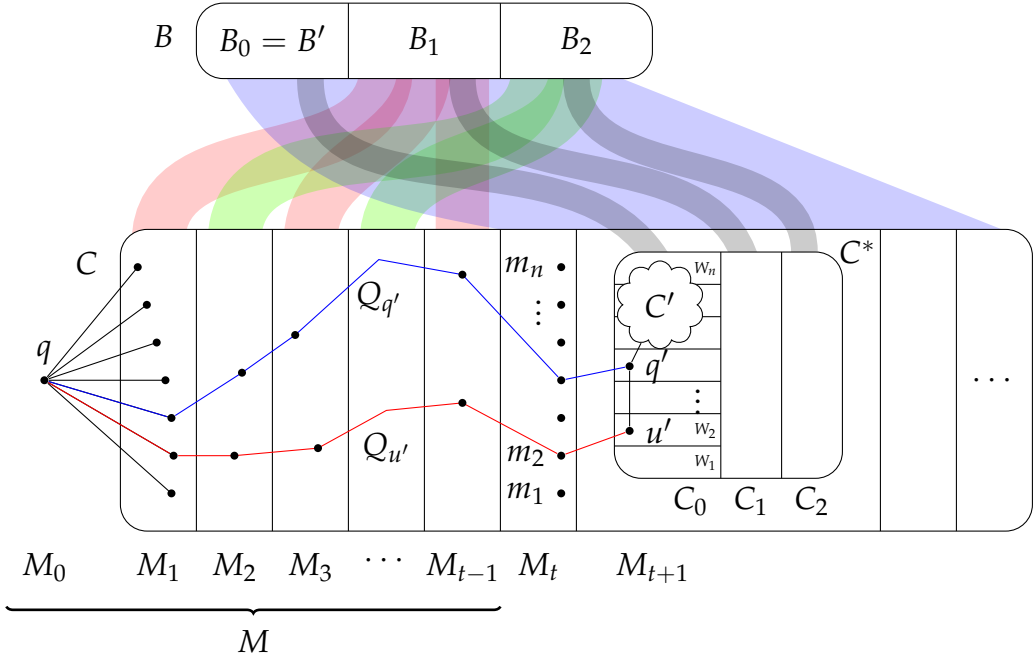
\begin{figure}
		\begin{center}
			\begin{tikzpicture}
				\draw [draw=red,line width=20pt,shorten >=-5pt,shorten <=-5pt, out=90,in=-90,opacity=.2] (0.5,4) to (3.5,6);
				\draw [draw=red,line width=20pt,shorten >=-5pt,shorten <=-5pt, out=90,in=-90,opacity=.2] (2.5,4) to (4,6);
				\draw [draw=red,line width=20pt,shorten >=-5pt,shorten <=-5pt, out=90,in=-90,opacity=.2] (4.5,4) to (4.5,6);
				\draw [draw=secondcolor,line width=20pt,shorten >=-5pt,shorten <=-5pt, out=90,in=-90,opacity=\secondopacity] (1.5,4) to (5.5,6);
				\draw [draw=secondcolor,line width=20pt,shorten >=-5pt,shorten <=-5pt, out=90,in=-90,opacity=\secondopacity] (3.5,4) to (6,6);
				\draw [draw=none,fill=blue,opacity=.2] (5,4) [out=170,in=-60] to (1.4,6)--(6.6,6)--(11.6,4)--cycle;

				\tikzstyle{v}=[circle, fill, inner sep=0pt, minimum size=3pt]
				\draw [fill=white,rounded corners=10pt] (1,6) rectangle (7,7);
				\node at (1,7) [label=below left:$B$]{};            
				\draw (3,6)--(3,7);
				\draw (5,6)--(5,7);
				\node at (2,6.5) {$B_0=B'$};
				\node at (4,6.5) {$B_1$};
				\node at (6,6.5) {$B_2$};
				\node [label=$q$,v] at (-1,2) (q) {};
				\draw [fill=white,rounded corners=10pt] (0,0) rectangle (12,4);
				\node at (0,4) [label=below left:$C$]{};
				\foreach \i in {1,2,3,4,5,6,10,11}
				{
					\draw (\i,0)--(\i,4);
				}
				\foreach \y in {1,2,3,4,5,6,7}{
					\pgfmathparse{cos(\y*25-50)}
					\let\xcoord\pgfmathresult
					\draw (q)    -- (.4+.3*\xcoord,0.5*\y) node (v\y)[v]{};
				}
				\foreach \y in {3,4,5,6}{
					\node at (5.8,0.5*\y) (m\y)[v] {};
				}
				\foreach \y in {1,2}{ 
					\node at (5.8,0.5*\y) (m\y)[v,label=left:$m_\y$] {};
				}
				\node at (5.8,0.5*7) (m7)[v,label=left:$m_n$] {};
				\node at (5.5,3) {$\vdots$};
				\node at (11.5,2) {$\cdots$};
				\node at (7,0)[label=below:$M_{t+1}$] {};
				\node at (-1,0)[label=below:$M_0$] {};
				\node at (0.5,0)[label=below:$M_1$] {};
				\node at (1.5,0)[label=below:$M_2$] {};
				\node at (2.5,0)[label=below:$M_3$] {};
				\node at (3.5,0)[label=below:$\cdots$] {};
				\node at (4.5,0)[label=below:$M_{t-1}$] {};
				\node at (5.5,0)[label=below:$M_{t}$] {};
				\draw [decorate,decoration={brace},line width=1pt] (5,-1)--(-1.5,-1) node[midway,label=below:$M$]{};
				
				\draw [line width=10pt, out=120,shorten <=-5pt,in=-90,opacity=.2] (7,3.7) to (2.5,6);
				\draw [line width=10pt, out=105,shorten <=-5pt,in=-90,opacity=.2] (8.2,3.7) to (4.5,6);
				\draw [line width=10pt, out=90,in=-90,opacity=.2] (9,3.7) to (6,6);
				\draw [rounded corners=10pt,fill=white] (6.5,.7)  rectangle (9.5,3.7);
				\draw (9.5-.8,.7)--(9.5-.8,3.7);
				\draw (9.5-.8*2,.7)--(9.5-.8*2,3.7);
				\foreach \i in {0,1,2}
				{
					\node at (6.5+0.8*\i+1,.4) {$C_\i$};
				}
				\node at (9.7,3.7) {$C^*$};
				
				\foreach \i in {1,2,3,4,5,6}
				{
					\draw (6.5, .7+3*\i/7) --(9.5-.8*2,.7+3*\i/7);
				}
				\foreach \i in {1,2}
				{
					\node at (7.7, .48+3*\i/7) {\tiny $W_\i$};
				}
				\node at (7.6, .54+3*3/7) {$\vdots$};
				\node at (7.7, .48+3) {\tiny $W_n$};
				\node at (6.7,.9+3/7) [v,label=right:$u'$] (u') {};
				\node at (6.7,.9+3*3/7) [v,label=right:$q'$] (q'){};
				
				\draw [draw=red] (q)--(v2)--(1.5,1) node[v]{} -- (2.6,1.1) node[v]{}--(3.5,1.6)node[label=below:$Q_{u'}$]{} --(4.5,1.7) node[v]{}--  (m2)--(u');
				\draw [draw=blue](q)--(v3)--(1.6,2.1) node[v]{} -- (2.3,2.6) node[v]{}--(3.4,3.6)node[label=below:$Q_{q'}$]{} --(4.5,3.4) node [v]{} -- (m4)--(q');
				\draw (u')--(q');
				\node[draw, fill=white, cloud, cloud puffs=10, minimum width=4pt, minimum height=4pt] at (7.1,3)(c') {$C'$};
				\draw (q')--(c');            
			\end{tikzpicture}
		\end{center}
		\caption{The situation in the proof of \cref{lem:ropeinduction} when $\chi(C_0)\ge c+3$.}\label{fig:part1}
	\end{figure}
	
	\begin{proof}
		For $i \ge 0$ let $M_i$ be the set of vertices in $C \cup \{q\}$ with $G[C \cup \{q\}]$-distance $i$ from $q$. Choose $t$ such that $\chi(M_{t+1}) \ge \lceil \chi(C)/2\rceil \ge 3c+9$. Since $3c+9 > 2$ and $\chi(M_0\cup M_1\cup M_2\cup M_3\cup M_4)\le 2$, it follows that $t \ge 4$.
		
		It follows that there exists some $C^*\subseteq M_{t+1}$ such that $G[C^*]$ is connected, the $G$-distance between $q$ and $C^*$ is at least $5$, and $\chi(C^*)\ge \chi(M_{t+1}) - \chi(N^{4}_G[q])\ge \chi(M_{t+1}) -2 \ge 3c+7$.
		Let $M=M_0 \cup \cdots \cup M_{t-1}$.
		Then, $G[M]$ is connected and anticomplete to $C^*$.
		
		Next we partition $B$ into three sets.
		Let $B_0$ be the set of vertices in $B$ with no neighbour in $M$.
		Let $B_1$ be the set of vertices $v$ in $B$ with a neighbour in $M$ such that the $G[M \cup \{v\}]$-distance between $q$ and $v$ is odd, and let $B_2$ be the set where this distance is even.
		For $i = 0,1,2$, let $C_i$ be the set of vertices in $C^*$ with a neighbour in $B_i$.
		Then, $C^*=C_0\cup C_1 \cup C_2$ and note that $B_i$ covers $C_i$ for each $i=0,1,2$.
		We must consider the cases that either $G[C_0]$ or $G[C_1]$ or $G[C_2]$ has large chromatic number.
		First we handle when $G[C_0]$ has large chromatic number.
		
		Suppose that $\chi(C_0)\ge c+3$.
		See \cref{fig:part1} for an illustration of the situation.
		Let $B'=B_0$.
		Take an ordering $m_1, \ldots , m_n$ of $M_t$, and for $1 \le i \le n$, let $W_i$ be the set of vertices $v \in C_0$ such that $v$ is adjacent to $m_i$ and nonadjacent to $m_1, \ldots, m_{i-1}$. Then $(W_1, \ldots, W_n)$ is a stable grading of $C_0$ since $G$ is triangle-free.
		
		By \cref{lem:earlier2} there exists a subset $C'$ of $C_0$ and an edge $u'q'$ of $G[C_0]$ with the following properties:
		\begin{itemize}
			\item $G[C'\cup \{q'\}]$ is connected,
			\item $\chi(C') \ge c$.
			\item $u'$ and $q'$ are earlier than every vertex in $C'$, and
			\item $u'$ has no neighbour in $C'$.
		\end{itemize}
		Let $Q_{u'}$ be an induced path in $G[M\cup M_{t} \cup \{u'\}]$ between~$u'$ and~$q$ of length~$t+1$ and similarly let $Q_{q'}$ be an induced path in $G[M\cup M_{t} \cup \{q'\}]$ between $q'$ and~$q$ of length~$t+1$.
		Observe that $Q_{u'}\cup \{q'\}$ is an induced path of length $t+2$ between~$q'$ and~$q$ since $G$ is triangle-free and there are no edges between $M$ and $M_{t+1}$.
		Of these two paths between~$q$ and $q'$, let $Q_0$ be the one of even length and let $Q_1$ be the one of odd length. This now provides the desired subsets $B'\subseteq B$, $C'\subseteq C$, vertex $q'\in C \setminus C'$, and paths $Q_0$, $Q_1$, so we may assume that $\chi(C_0) < c+3$.

		\begin{figure}
			\begin{center}
				\begin{tikzpicture}
					\draw [draw=red,line width=20pt,shorten >=-5pt,shorten <=-5pt, out=90,in=-90,opacity=.2] (0.5,4) to (3.5,6);
					\draw [draw=red,line width=20pt,shorten >=-5pt,shorten <=-5pt, out=90,in=-90,opacity=.2] (2.5,4) to (4,6);
					\draw [draw=red,line width=20pt,shorten >=-5pt,shorten <=-5pt, out=90,in=-90,opacity=.2] (4.5,4) to (4.5,6);
					\draw [draw=secondcolor,line width=20pt,shorten >=-5pt,shorten <=-5pt, out=90,in=-90,opacity=\secondopacity] (1.5,4) to (5.5,6);
					\draw [draw=secondcolor,line width=20pt,shorten >=-5pt,shorten <=-5pt, out=90,in=-90,opacity=\secondopacity] (3.5,4) to (6,6);
					\draw [draw=none,fill=blue,opacity=.2] (5,4) [out=170,in=-60] to (1.4,6)--(6.6,6)--(11.6,4)--cycle;

					\tikzstyle{v}=[circle, fill, inner sep=0pt, minimum size=3pt]
					\draw [fill=white,rounded corners=10pt] (1,6) rectangle (7,7);
					\node at (1,7) [label=below left:$B$]{};            
					\draw (3,6)--(3,7);
					\draw (5,6)--(5,7);
					\node at (2,6.5) {$B_0$};
					\node at (4,6.5) {$B_1$};
					\node at (6,6.5) {$B_2$};
					\node [label=left:{$q=m_1$},v] at (-1,2) (q) {};
					\draw [fill=white,rounded corners=10pt] (0,0) rectangle (12,4);
					\node at (0,4) [label=below left:$C$]{};
					\foreach \i in {1,2,3,4,5,6,10,11}
					{
						\draw (\i,0)--(\i,4);
					}
					\foreach \y in {1,2,3,4,5,6,7}{
						\draw (q)    -- (.3,0.5*\y) node (v\y)[v]{};
						\node at (1.5,0.5*\y) (w\y)[v]{};
						\node at (2.3,0.5*\y) (ww\y)[v]{};
						\node at (3.5,0.5*\y) [v]{};
						\node at (4.5,0.5*\y) (z\y)[v]{};
						
					}
					\node at (v1) [label=below:\tiny $m_2$,label distance=0pt,inner sep=0pt]{};
					\node at (v2) [label=below:\tiny $m_3$,label distance=0pt,inner sep=0pt]{};
					\node at (w1) [label=below:\tiny $m_{\abs{M_1}+2}$,label distance=0pt,inner sep=0pt]{};
					\node at (z7) [label=above:\tiny $m_{n}$,label distance=0pt,inner sep=0pt]{};
					
					\node at (11.5,2) {$\cdots$};
					\node at (7,0)[label=below:$M_{t+1}$] {};
					\node at (-1,0)[label=below:$M_0$] {};
					\node at (0.5,0)[label=below:$M_1$] {};
					\node at (1.5,0)[label=below:$M_2$] {};
					\node at (2.5,0)[label=below:$M_3$] {};
					\node at (3.5,0)[label=below:$\cdots$] {};
					\node at (4.5,0)[label=below:$M_{t-1}$] {};
					\node at (5.5,0)[label=below:$M_{t}$] {};
					\draw [decorate,decoration={brace},line width=1pt] (5,-1)--(-1.5,-1) node[midway,label=below:$M$]{};
					
					\draw [line width=10pt, out=120,shorten <=-5pt,in=-90,opacity=.2] (7,3.7) to (2.5,6);
					\draw [line width=10pt, out=105,shorten <=-5pt,in=-90,opacity=.2] (8.2,3.7) to (4.5,6);
					\draw [line width=10pt, out=90,in=-90,opacity=.2] (9,3.7) to (6,6);
					\draw [rounded corners=10pt,fill=white] (6.5,.7)  rectangle (9.5,3.7);
					\draw (9.5-.8,.7)--(9.5-.8,3.7);
					\draw (6.5+.8,.7)--(6.5+.8,3.7);
					\node at (6.5+.4,.4) {$C_0$};
					\node at (8,.4) {$C_1$};
					\node at (6.5+0.8*2+1,.4) {$C_2$};
					\node at (9.7,3.7) {$C^*$};
					
					\foreach \i in {1,2,3,4,5,6}
					{
						\draw (6.5+0.8, .7+3*\i/7) --(9.5-.8*2+0.8,.7+3*\i/7);
					}
					\foreach \i in {1,2}
					{
						\node at (7.5, .48+3*\i/7) {\tiny $W_\i$};
					}
					\node at (7.6, .54+3*3/7) {$\vdots$};
					\node at (7.5, .48+3) {\tiny $W_n$};
					\node at (6.7+1.4,.9+3/7) [v,label=right:$u'$] (u') {};
					\node at (6.7+1.4,.9+3*3/7) [v,label=right:$q'$] (q'){};
					
					\draw [draw=red] (u')--(q');
					\node[draw, fill=white, cloud, cloud puffs=10, minimum width=4pt, minimum height=4pt] at (7.1+1,3)(c') {$C'$};
					\draw (q')--(c');
					\draw [draw=red](u') [out=140,in=-60] to 
					node[pos=.45,sloped,above] {\tiny $Q_{2-h}=Q_1$}
					(3.3,6.2) node[v,label=$b_{u'}$]{}--(v1) node [label=right:{\tiny $m_{i_{u'}}$},label distance=0pt,inner sep=0pt]{}--(q);
					\draw [draw=blue] (q') [out=140,in=-70] to 
					node[pos=.6,sloped,above] {\tiny $Q_{h-1}=Q_0$}(4.5,6.2)
					node[v,label=$b_{q'}$]{}--(ww3) node [label=right:{\tiny $m_{i_{q'}}$},label distance=0pt,inner sep=0pt]{}--(w4)--(v6)--(q) ;
				\end{tikzpicture}
			\end{center}
			\caption{The situation in the proof of \cref{lem:ropeinduction} when $\chi(C_0)<c+3$ and $\chi(C_1)\ge c+3$.}\label{fig:part2}
		\end{figure}
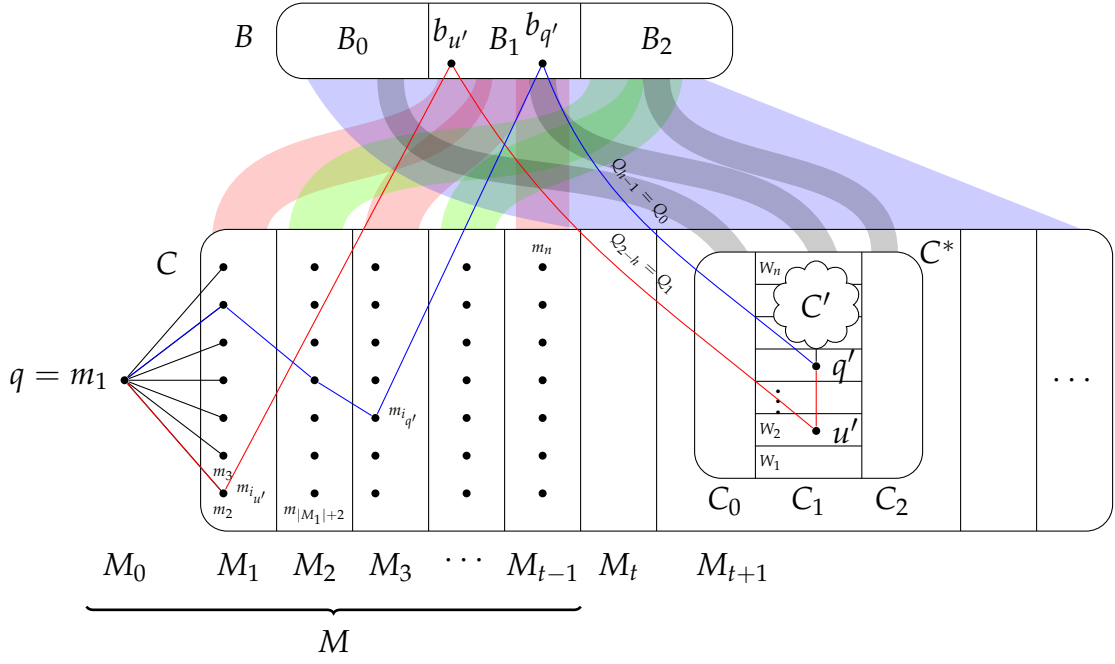
		
		Since $\chi(C_0) < c + 3$, it follows that for some $h\in \{1,2\}$, we have that $\chi(C_h) \ge c+3$. See \cref{fig:part2} for an illustration of the situation.
		Take an ordering $m_1, \ldots , m_n$ of $M$ so that for $1\le i < j\le n$, we have that the $G[M]$-distance between $q$ and $m_i$ is at most the $G[M]$-distance between $q$ and $m_j$. For $1 \le i \le n$ let $W_i$ be the set of vertices $v \in C_h$ such that $v$ has a neighbour in $B_h$ adjacent to $m_i$ and no neighbour in $B_h$ adjacent to one of $m_1, \ldots , m_{i-1}$.
		Then $(W_1, \ldots , W_n)$ is a stable grading of $C_h$ since $G$ has odd girth at least~$7$ and $W_i\subseteq N_G^2(m_i)$ for each $1\le i \le n$. 
		
		By \cref{lem:earlier2}, there exists a subset $C'$ of $C_h$ and an edge $u'q'$ of $G[C_h]$ with the following properties:
		\begin{itemize}
			\item $G[C'\cup \{q'\}]$ is connected,
			\item $\chi(C') \ge c$.
			\item $u'$ and $q'$ are earlier than every vertex in $C'$, and
			\item $u'$ has no neighbour in $C'$.
		\end{itemize}
		Let $i_{u'}$ be such that $u'\in W_{i_{u'}}$ and similarly, let $i_{q'}$ be such that $q'\in W_{i_{q'}}$.
		Let $k=\max\{i_{u'}, i_{q'}\}$.
		Let $B'=B_h\setminus N(\{m_1, \ldots , m_k\})$.
		Since $u'$ and $q'$ are earlier than each vertex of $C'$, it follows that $B'$ covers $C'$.
		Let $b_{u'}$ be a vertex of $B_h\setminus B'$ adjacent to both $m_{i_{u'}}$ and $u'$, and similarly, let $b_{q'}$ be a vertex of $B_h\setminus B'$ adjacent to both $m_{i_{q'}}$ and $q'$.
		Since $G$ is triangle-free, $b_{u'}$ is non-adjacent to $q'$.
		Since the $G$-distance between $q$ and $C^*$ is at least $5$, neither $b_{u'}$ nor $b_{q'}$ is in $N^3(q)$.
		Let $Q_{2-h}$ be an induced path in $G[M\cup \{b_{u'}, u', q'\}]$ between $q$ and $q'$
		obtained by extending a shortest path from $q$ to $b_{u'}$ in $G[M\cup\{b_{u'}\}]$ by two edges $b_{u'}u'$ and $u'q'$.
		Then, $Q_{2-h}$ has even length if $2-h=0$ and odd length if $2-h=1$.
		Similarly, let $Q_{h-1}$ be an induced path in $G[M\cup \{b_{q'}, q'\}]$ between $q$ and~$q'$
		obtained by extending a shortest path from $q$ to $b_{q'}$ in $G[M\cup\{b_{q'}\}]$ by an edge $b_{q'}q'$.
		Then, $Q_{h-1}$ has even length if $h-1=0$ and odd length if $h-1=1$.
		This now provides the desired subsets $B'\subseteq B$, $C'\subseteq C$, a vertex $q'\in C \setminus C'$, and induced paths $Q_0$, $Q_1$.
	\end{proof}
	
	A \emph{broken $r$-arithmetic rope} is a graph consisting of paths $Q_{1,1}, Q_{1,2}, \ldots , Q_{r,1}, Q_{r,2}$ with ends contained in a vertex set $\{q_1,\ldots , q_{r+1}\}$ so that 
	\begin{itemize}
		\item for every $1\le i \le r$, $Q_{i,1}$ is an odd-length path between $q_i$ and $q_{i+1}$,
		\item for every $1\le i \le r$, $Q_{i,2}$ is an even-length path between $q_i$ and $q_{i+1}$,
		\item for every $(h_1, \ldots , h_r) \in \{1,2\}^r$, $Q_{1,h_1} \cup Q_{2,h_2} \cup \cdots \cup  Q_{r,h_r}$ is an induced path, and
		\item the vertices $q_1,\ldots , q_{r+1}$ are pairwise at $G$-distance at least $5$.
	\end{itemize}
	We denote such a broken $r$-arithmetic rope by $(q_i,Q_{i,1},Q_{i,2})_{i=1}^r$ and say that it has \emph{end}~$q_{r+1}$.
	By applying \cref{lem:ropeinduction} a total of $r$ times, we can find a broken $r$-arithmetic rope.
	
	\begin{lemma}\label{lem:broken}
		Let $r\ge 1$, 
		let $G$ be a graph with odd girth at least $11$, and let $B,C\subseteq V(G)$ be disjoint vertex subsets with $B$ covering $C$. Let $q_1\in V(G)\setminus C$ be a vertex such that $G[C\cup \{q_1\}]$ is connected and $\chi(G[C]) \ge 6^rc+\frac{17}{5}(6^r-1)$.
		Then there exists $B'\subseteq B$, $C'\subseteq C$, $q_2,q_3,\ldots,q_{r+1}\in C\setminus C'$ and a broken $r$-arithmetic rope $(q_i,Q_{i,1},Q_{i,2})_{i=1}^r$ 
		with end $q_{r+1}$
		such that
		\begin{itemize}
			\item $G[C'\cup \{q_{r+1}\}]$ is connected,
			\item $\chi(C') \ge c$,
			\item $B'$ covers $C'$,
			\item $B'$ is anticomplete to $(V(Q_{i,1}) \cup V(Q_{i,2}) )   \setminus N^2[q_{i+1}]$ for all $i=1,2,\ldots,r$,
			\item $(V(Q_{i,1}) \cup V(Q_{i,2}) )  \setminus N^2[q_{i+1}]\subseteq C\cup\{q_1\}$ for all $i=1,2,\ldots,r$,
			\item $C'$ is anticomplete to $\left(\bigcup_{i=1}^r (V(Q_{i,1}) \cup V(Q_{i,2}) ) \right)  \setminus \{q_{r+1}\}$, and
			\item the $G$-distance between $\{q_1, \ldots , q_r\}$ and $C'\cup \{q_{r+1}\}$ is at least $5$.
		\end{itemize}
	\end{lemma}
	\begin{proof}
		We proceed by the induction on $r$. The base case $r=1$ follows from \cref{lem:ropeinduction}. So, let us assume that $r>1$ and
		that the lemma holds for $r-1$.
		Thus there exists $B_{r-1}\subseteq B$, $C_{r-1}\subseteq C$, $q_2,q_3,\ldots,q_{r}\in C\setminus C_{r-1}$, and a broken $(r-1)$-arithmetic rope $(q_i,Q_{i,1},Q_{i,2})_{i=1}^{r-1}$ satisfying the above properties with $c':=6c+17$.
		We remark that $\chi(C_{r-1})\ge 6c+17$.
		By applying \cref{lem:ropeinduction} with $B_{r-1}$, $C_{r-1}$, and $q_{r}$, 
		we deduce that there exists $B':=B_{r}\subseteq B_{r-1}$, $C':=C_{r}\subseteq C_{r-1}$, $q_{r+1}\in C_{r-1}\setminus C_{r}$, and two induced paths $Q_{r,1}$, $Q_{r,2}$ between $q_{r}$ and $q_{r+1}$ in $G[(C_{r-1} \setminus C_{r}) \cup (B_{r-1} \setminus B_{r}) \cup \{q_{r}\}]$ satisfying the properties of \cref{lem:ropeinduction}.
		It is easy to check that the above seven properties hold.   
		
		It remains to check that $(q_i,Q_{i,1},Q_{i,2})_{i=1}^r$ is indeed a broken $r$-arithmetic rope.
		The first, second, and last conditions are immediate.
		To see the third condition, suppose that there is $(h_1, \ldots , h_r) \in \{1,2\}^r$ such that $Q_{1,h_1} \cup Q_{2,h_2} \cup \cdots \cup  Q_{r,h_r}$ is not an induced path.
		By the induction hypothesis, $Q_{1,h_1} \cup Q_{2,h_2} \cup \cdots \cup  Q_{r-1,h_{r-1}}$ is an induced path. Thus, there is an edge joining a vertex $x\in V(Q_{a,h_i}) \setminus\{q_{r}\}$ and a vertex $y\in V(Q_{b,h_r}) \setminus \{q_{r}\}$ for some $i\in\{1,2,\ldots,r-1\}$ and $a,b\in \{0,1\}$.
		Since $C_{r-1}$ is anticomplete to $(V(Q_{i,1})\cup V(Q_{i,2})) \setminus\{q_r\}$, it follows that $y\in (B_{r-1}\cap N^2(q_{r+1}))\setminus N^3(q_r)$. 
		Since $B_{r-1}$ is anticomplete to $(V(Q_{i,1})\cup V(Q_{i,2})) \setminus N^2[q_{i+1}]$, it follows that $x\in N^2[q_{i+1}]$.
		Then we find a path of length at most $3$ from $q_{i+1}$ to $y$.
		Note that
		$y$ has a neighbour in $C_{r-1}$ on $Q_{b,h_r}$
		because $V(Q_{b,h_r})\cap B_{r-1}\subseteq N^2(q_{r+1})$.
		This implies that if $i<r-1$, then we obtain a path of length~$4$ from $q_{i+1}$ to $C_{r-1}$, 
		contradicting the induction hypothesis that the $G$-distance between $\{q_1, \ldots , q_{r-1}\}$ and $C_{r-1}$ is at least~$5$.
		Thus, $i=r-1$
		and the $G$-distance from $q_{r}$ to $y$ is at most $3$, a contradiction.
	\end{proof}

	To prove \cref{thm:ropefinding} we shall now find a certain broken $r$-arithmetic rope and create an $r$-arithmetic rope from it by fixing it with an additional path between $q_{r+1}$ and~$q_1$ (and extending $Q_{r,1}$, $Q_{r,2}$ along this path).
	
	\begin{proof}[Proof of \cref{thm:ropefinding}]
		We may assume that $G[X]$ is connected since we can always restrict to a connected component with maximum chromatic number.
		Choose some vertex in $X$, and for $i \ge 0$, let $L_i$ be the set of vertices with $G[X]$-distance $i$ from the vertex. There exists $s$ such that $\chi(L_{s+1}) \ge \lceil \chi(G)/2\rceil \ge 6^r \cdot 3 + \frac{17}{5}(6^r-1)$. Since $G$ has odd girth at least $11$, $\chi(L_0 \cup \cdots \cup L_{4}) \le 2$, so it follows that $s \ge 4$. Let $C$ be the vertex set of a connected component of $G[L_{s+1}]$ with maximum chromatic number.
		Let $q_1$ be a vertex of $L_s$ with a neighbour in $C$, and let $B=L_s$.
		
		By \cref{lem:broken}, there exists $B'\subseteq B$, $C'\subseteq C$, $q_2,q_3,\ldots,q_{r+1}\in C\setminus C'$ and a broken $r$-arithmetic rope $(q_i,Q_{i,1},Q_{i,2})_{i=1}^r$ with end $q_{r+1}$ such that
		\begin{itemize}
			\item $G[C'\cup \{q_{r+1}\}]$ is connected,
			\item $\chi(G[C']) \ge 3$,
			\item $B'$ covers $C'$,
			\item $B'$ is anticomplete to $(V(Q_{i,1}) \cup V(Q_{i,2}) )   \setminus N^2[q_{i+1}]$ for all $i=1,2,\ldots,r$,
			\item $(V(Q_{i,1}) \cup V(Q_{i,2}) )  \setminus N^2[q_{i+1}]\subseteq C\cup\{q_1\}$ for all $i=1,2,\ldots,r$,
			\item $C'$ is anticomplete to $\left(\bigcup_{i=1}^r (V(Q_{i,1}) \cup V(Q_{i,2}) ) \right)  \setminus \{q_{r+1}\}$, and
			\item the $G$-distance between $\{q_1, \ldots , q_r\}$ and $C'\cup \{q_{r+1}\}$ is at least $5$.
		\end{itemize}
		Since $\chi(G[C']) \ge 3$, there exists a vertex $x\in C'$ with $G$-distance at least $5$ from $q_{r+1}$ and also from $q_1, \ldots , q_r$ since $x\in C'$.
		Let $b\in B'$ be a vertex adjacent to $x$.
		Let $P_1$ be a shortest induced path between $q_{r+1}$ and $b$
		in $G[\{q_{r+1},b\}\cup C']$.
		Since the $G$-distance between~$q_{r+1}$ and $x$ is at least $5$, the length of $P_1$ is at least $4$.
		
		Let $a_1\in L_{s-1}$ be a vertex adjacent to $q_1$ and let $a_2\in L_{s-1}$ be a vertex adjacent to~$b$.
		Observe that the vertices $q_1, \ldots , q_{r+1}, x$ are pairwise at $G$-distance at least $5$, $B'$ is anticomplete to $(V(Q_{i,1}) \cup V(Q_{i,2}) ) \setminus N^2[q_{i+1}]$ for every $i=1,2,\ldots,r$, and
		$(V(Q_{i,1}) \cup V(Q_{i,2}) )  \setminus N^2[q_{i+1}]\subseteq C\cup\{q_1\}$ for all $i=1,2,\ldots,r$.
		Therefore the only neighbour of~$a_1$ in $V(P_1)\cup \{a_2\} \cup \bigcup_{i=1}^r (V(Q_{i,1}) \cup V(Q_{i,2}) )$ is $q_1$.
		Furthermore, since 
		the $G$-distance between $x$ and $q_i$ is at least $5$ for every $i=1,2,\ldots,{r+1}$ and 
		$C'$ is anticomplete to $\left(\bigcup_{i=1}^r (V(Q_{i,1}) \cup V(Q_{i,2}) ) \right)  \setminus \{q_{r+1}\}$, it follows that $P_1ba_2$ is an induced path and $V(P_1)\cup \{a_2\}\setminus\{q_{r+1}\}$ is anticomplete to $\bigcup_{i=1}^r (V(Q_{i,1}) \cup V(Q_{i,2}) ) \setminus \{q_{r+1}\}$.
		
		Let $P_2$ be an induced path between $a_1$ and $a_2$ in $G[\{a_1,a_2\} \cup \bigcup_{i=0}^{s-2} L_i]$.
		Clearly $V(P_2)\setminus \{a_1,a_2\}$ is anticomplete to $V(P_1)\cup \bigcup_{i=1}^r (V(Q_{i,1}) \cup V(Q_{i,2}) )$.
		Let $P$ be the induced path $P_1ba_2P_2a_1q_1$ between $q_{r+1}$ and $q_1$.
		Then we form an $r$-arithmetic rope from the broken $r$-arithmetic rope $(q_i,Q_{i,1},Q_{i,2})_{i=1}^r$ with end $q_{r+1}$ by extending $Q_{r,1}$, $Q_{r,2}$ along~$P$,  possibly swapping their labels depending on the parity of the length of $P$.
	\end{proof}
	
	We actually only need \cref{thm:ropefinding} in the $r=5$ case, in which case we have the following bound.
	
	\begin{theorem}
		\label{thm:5ropefinding}
		Every graph $G$ with odd girth at least $11$ and $X\subseteq V(G)$ with $\chi(X) \ge 99525$ contains a $5$-arithmetic rope in $X$.
	\end{theorem}

	Let us remark that with more technical arguments, one can replace the odd girth condition in \cref{thm:ropefinding} with the conditions that $\omega(G)\le \omega$, every induced subgraph~$H$ of $G$ with $\omega(H)<\omega$ has chromatic number at most~$\tau$, and $\chi^{4}(G)=\max_{v\in V(G)}\{\chi(N^{4}[v])\}$ is bounded,  giving a much worse bound for the chromatic number that now depends on $\omega$, $\tau$, $\chi^{4}(G)$, as well as $r$.
	One can even alter the definition of arithmetic ropes to say that the length of the paths $Q_{i,1}, Q_{i,2}$ differ by $1 \pmod{m}$ rather than just $1 \pmod{2}$.
	In particular, such arithmetic ropes can then be made to contain an induced cycle of length $\ell \pmod{m}$.
	This yields another proof of the key step of the proof of a theorem of Scott and Seymour \cite{scott2019induced} that the class of graphs without an induced cycle of length $\ell \pmod{m}$ is $\chi$-bounded.

	\section{Odd wheel t-minors} 
	\label{sec:oddwheel}

	This section is dedicated to proving \Cref{thm:technical-main} (which then implies \cref{thm:main} and thus \cref{thm:hperfect}).
	Our main tool is \cref{thm:5ropefinding}, which was proven in the last section.
	
	We begin with some lemmas showing that certain graph structures contain odd wheels as t-minors.
	
	\begin{lemma}\label{lem:oddwheel}
		Let $G$ be a graph consisting of an induced odd cycle $C$ and a single additional vertex $v$ adjacent to at least three vertices $a$, $b$, $c$ that appear on $C$ and that partition $C$ into three odd-length paths.
		Then $G$ contains an odd wheel as a t-minor.
	\end{lemma}
	\begin{proof}
		If $\abs{V(C)}=3$, then $G=K_4$, and thus the lemma holds.
		So, we may assume that $\abs{V(C)}>3$ and we proceed inductively.
		If $v$ is adjacent to every vertex of $C$, then $G$ is an odd wheel, so we may assume that some vertex $u$ of $C$ is not adjacent to~$v$.
		By the assumption, $u$ is adjacent to at most one of $a$, $b$, and $c$.
		Then we apply a t-contraction at~$u$ to obtain a smaller graph $G'$ that still consists of an odd induced cycle $C'$ and a single additional vertex $v$ adjacent to at least three vertices $a$, $b$, $c$ that appear in order with on~$C'$ and partition $C'$ into three odd-length paths.
		By the inductive hypothesis, $G'$ (and hence $G$) contains an odd wheel as a t-minor, as desired.
	\end{proof}
	
	The following lemma extends \cref{lem:oddwheel}.

	\begin{lemma}\label{lem:oddwheel3}
		Let $G$ be a graph and 
		let $X$ be a subset of $V(G)$ such that 
		$G[X]$ contains an $r$-arithmetic rope $(q_i,Q_{i,1},Q_{i,2})_{i=1}^r$. 
		If $G- X$ is a connected bipartite graph with a bipartition $(A,B)$
		such that 
		no vertex in $A$ has neighbours in $X$ and 
		at least three vertices of $\{q_1,q_2,\ldots,q_r\}$ have neighbours in $B$, 
		then 
		$G$ contains an odd wheel as a t-minor.
	\end{lemma}
	
	\begin{proof}
		We proceed by induction on $\abs{A}$.
		Observe that 
		$G[X]$ has an odd induced cycle~$C$ containing three vertices of $\{q_1,q_2,\ldots,q_k\}\cap N(B)$ that partition $C$ into three odd-length paths.
		If $A=\emptyset$, then by applying \cref{lem:oddwheel} to $C$ with the unique vertex in $B$, we deduce that $G$ contains an odd wheel as a t-minor. 
		If $A$ contains a vertex~$v$, then we perform a t-contraction at $v$ to obtain a t-minor $G'$. Observe that $G'$ is the graph obtained from $G$ by contracting all the edges incident with $v$ and so still satisfies the assumption of the lemma. By the induction hypothesis, $G'$ contains an odd wheel as a t-minor and so does~$G$.
	\end{proof}
	
	We require one more simple lemma on graphs with large odd girth containing connected bipartite induced subgraphs that contain a distinguished set $S$ of vertices.
	We remark that there are similar Ramsey theorems on finding induced trees containing an increasing number (increasing with $\abs{S}$) of vertices of a distinguished set $S$ \cite{abrishami2024induced,davies2022vertex}.
	Such Ramsey results can be used in part to strengthen \cref{thm:technical-main} to triangle-free graphs (at the cost of much worse bounds).
	
	\begin{lemma}\label{lem:bipartite}
		Let $G$ be a connected graph with odd girth at least $2g+1$ and let $S$ be a stable set of~$G$ with $\abs{S}\le 2g$.
		Then $G$ contains a connected bipartite induced subgraph $H$ with $S\subseteq V(H)$.
	\end{lemma}
	
	\begin{proof}
		Let $H$ be a connected induced subgraph of $G$ with $S\subseteq V(H)$, and with $V(H)$ minimal subject to this.
		By minimality, for every $v\in V(H) \setminus S$, $H-v$ is disconnected and each connected component contains a vertex of $S$.
		If $H$ is a tree, then the lemma follows, so we may assume otherwise.
		Consider a cycle $C$ of $H$.
		For each $v\in V(C) \setminus S$, there is a connected component $H_v$ of $H-v$ containing a vertex of $S$ and not containing a vertex of $C$. Then, the induced subgraphs $(H_v : v\in V(C) \setminus S)$ are vertex-disjoint, and each contains a vertex of~$S$.
		Therefore, $\abs{V(C)}\le \abs{S}\le 2g$.
		So, $C$ is not an odd cycle since $G$ has odd girth at least $2g+1$.
		Hence $H$ is bipartite, as desired.
	\end{proof}
	
	We are now ready to prove \Cref{thm:technical-main} (which then implies \cref{thm:main} and thus \cref{thm:hperfect}).
	
	\begin{proof}[Proof of \Cref{thm:technical-main}.]
		Let $G$ be a graph with odd girth at least $11$ and chromatic number at least $199049$.
		
		Let $v$ be a vertex of $G$ in a connected component with maximum chromatic number.
		Let $L_0=\{v\}$ and for each positive integer $i$, let $L_i=N^i(v)$.
		Then there exists a positive integer $t$ such that $\chi(L_t) \ge \lceil 199049/2\rceil = 99525$.
		Since $\chi(N^3[v]) \le 2 < 99525$, we have that $t\ge 4$.
		
		Since $\chi(L_t) \ge 99525$, by \cref{thm:5ropefinding} there is a $5$-arithmetic rope $(q_i,Q_{i,1},Q_{i,2})_{i=1}^5$ of~$G$ contained in~$L_t$.
		For each $1\le i \le 5$, let $x_i$ be a vertex of $L_{t-1}$ adjacent to $q_i$, and let $y_i$ be a vertex of $L_{t-2}$ adjacent to $x_i$.
		Since the vertices $q_1, \ldots , q_5$ are pairwise at distance at least~$5$ in $G$, the vertices $x_1, \ldots , x_5 , y_1, \ldots , y_5$ are all distinct
		and each of $x_1, x_2, x_3, x_4, x_5$ has degree $1$ in  $G[\{x_1, \ldots , x_5 , y_1, \ldots , y_5\}]$.
		Now, let $H^*$ be the induced subgraph of $G$ on vertex set $\{x_1, \ldots, x_5\} \cup \{y_1, \ldots , y_5\} \cup \bigcup_{i=0}^{t-3} L_i$. Then, $H^*$ is connected and each of the vertices $x_1 , \ldots , x_5$ has degree one in $H^*$.
		
		Since $H^*$ has odd girth at least~$7$, by \cref{lem:bipartite}, $H^*$ contains a connected bipartite induced subgraph~$H'$ with $\{x_1, \ldots , x_5\} \subseteq V(H')$.
		Let $A$, $B$ be the bipartition of~$H'$.
		Without loss of generality,
		we may assume that $B$ contains at least three vertices $a^*,b^*,c^*$ of $\{x_1, \ldots , x_5\}$ that are pairwise at an even distance in $H'$.
		Since the vertices $x_1, \ldots , x_5$ all have degree 1 in $H'$, it follows that $H=H'- (\{x_1, \ldots , x_5\} \setminus \{a^*, b^*, c^*\})$ is connected.
		No vertex of $V(H)\setminus \{a^*, b^*, c^*\}$ has a neighbour in the arithmetic rope $(q_i,Q_{i,1},Q_{i,2})_{i=1}^5$ since $V(H)\setminus \{a^*, b^*, c^*\} \subseteq \bigcup_{i=0}^{t-2} L_i$ and $(q_i,Q_{i,1},Q_{i,2})_{i=1}^5$ is contained in $L_t$.
		Hence $G$ contains an odd wheel as a t-minor by \cref{lem:oddwheel3}.
	\end{proof}

	\section{Concluding remarks and open problems}
	\label{sec:conclusion}

	As remarked in \cref{sec:intro}, we optimised the proof of \cref{thm:main} for simplicity, rather than for the best bound.
	With more delicate and technical arguments, it is possible to significantly improve.
	This would still leave a large gap between our bound and the lower bound of 4 for the chromatic number of t-perfect graphs.
	A good milestone for narrowing the gap would be to improve our upper bound to at most 1000.
	The only forbidden t-minors we used were odd wheels, so considering further forbidden t-minors such as even M\"{o}bius ladders \cite{shepherd-lehman} or working directly with the polytope definition of t-perfection may allow for better bounds. The complete list of forbidden t-minors for t-perfection remains a major open problem. 
	
	While not all t-perfect graphs are $3$-colourable~\cite{schrijver2003combinatorial,benchetrit-PhD,benchetrit20164critical}, the graphs in \cref{fig:laurent-seymour} are the only known minimal counterexamples. Both counterexamples are relatively \say{uncomplicated} graphs, both contain at most $10$ vertices and both are complements of line graphs of well-known graphs. 
	In \cite{benchetrit20164critical}, Benchetrit showed that the only 4-critical $P_6$-free t-perfect graphs are the graphs in \cref{fig:laurent-seymour}, 
	but there may well be other minimal counterexamples that include an induced path on six vertices. 
	It is natural to ask if these are the only counterexamples, or if there are infinitely many counterexamples.

	\begin{problem}
		Are there infinitely many $4$-critical t-perfect graphs?
	\end{problem}

	We remark that if there are only finitely many $4$-critical t-perfect graphs, then this would yield a polynomial time algorithm for determining whether a t-perfect graph is $3$-colourable. It is NP-Complete to decide whether a graph is 3-colourable in general.

	Lov{\'a}sz \cite{lovasz1972normal} proved that the complement of every perfect graph is a perfect graph.
	This is not the case for h-perfect graphs, however the class \emph{$\overline{h}$-perfect} of complements of h-perfect graphs is still a natural class of graphs that is equivalently defined by a polytope.
	Observe that a graph~$G$ is \emph{$\overline{h}$-perfect} if its clique set polytope is equal to 
	\begin{align*}
		\overline{\hstab{G}}=\{ x\in \mathbb{R}^{V(G)}\colon 
		& x(S)\le 1 \text{ for every stable set $S$,}\\
		& x(C)\le \lfloor \abs{V(C)}/2\rfloor 
		\text{ for every odd anti-hole $C$,}\\
		& x(v)\ge 0 \text{ for every vertex $v$.} \}. 
	\end{align*}

	As with the class of h-perfect graphs, we prove that $\overline{h}$-perfect graphs are also $\chi$-bounded.
	We prove that $\overline{h}$-perfect graphs are quadratically $\chi$-bounded.
	We remark that a (doubly exponential) $\chi$-bound also follows from \cref{thm:hperfect} and a theorem of Scott and Seymour~\cite{complementation-conjecture}.
	
	\begin{lemma}\label{lem:antihole}
		For every integer $k\ge 3$, 
		$\overline{C_{2k+1}}$ is not h-perfect.
	\end{lemma}
	\begin{proof}
		First, observe that $\frac1k \mathbf{1}\in \hstab{\overline{C_{2k+1}}}$.
		So if it is h-perfect, then $\frac1k \mathbf{1}\in \ssp{\overline{C_{2k+1}}}$ and therefore there exist stable sets $S_1$, $S_2$, $\ldots$, $S_m$ with positive weights $\lambda_1$, $\lambda_2$, $\ldots$, $\lambda_m$ such that $\sum_{i=1}^m \lambda_i=1$ and $\sum_{i=1}^m \lambda_i \mathbf{1}_{S_i} = \frac1k \mathbf{1}$. Since each stable set has size at most $2$, we deduce that $\frac{2k+1}{k} = \sum_{i=1}^m \lambda_i \abs{S_i} \le 2 \sum_{i=1}^m\lambda_i = 2$, a contradiction.
	\end{proof}
	
	\begin{theorem}\label{thm:hcomplement}
		Every $\overline{h}$-perfect graph~$G$ is $\omega(G)+1 \choose 2$-colourable.
	\end{theorem}
	
	\begin{proof}
		We proceed by induction on $\abs{V(G)}$.
		Let $\omega=\omega(G)$. We may assume that $\omega>1$.
		Let $v$ be a vertex of $G$.
		By \cref{lem:antihole}, $G$ contains no odd hole of length at least~$7$.
		Observe that $G-N[v]$ contains no odd anti-hole since odd wheels are not h-perfect by \cref{lem:oddwheelforbid}.
		So, by the strong perfect graph theorem \cite{CRST2006}, 
		$G-N[v]$ is perfect and therefore 
		$\chi(G-N(v)) = \chi(G- N[v])= \omega(G- N[v]) \le \omega$.
		Clearly $\omega(G[N(v)]) < \omega$, so by the induction hypothesis,
		$\chi(G)\le \chi(G- N(v)) + \chi(G[N(v)])\le \omega+\binom{\omega}{2}=\binom{\omega+1}{2} $.
	\end{proof}

	\cref{thm:hcomplement} is tight for triangle-free graphs since $C_5$ is $\overline{h}$-perfect. More generally, pairwise complete copies of $C_5$ are $\overline{h}$-perfect and this shows that for each positive integer~$\omega$, there is an $\overline{h}$-perfect graph with clique number~$\omega$ and chromatic number $\lfloor \frac{3}{2}\omega \rfloor$.
	So, a $\chi$-bounding function of the form $\omega(G) + c$ as in \cref{thm:hperfect} for h-perfect graphs is not possible for $\overline{h}$-perfect graphs.
	We conjecture that the quadratic bound in \cref{thm:hcomplement} can be improved to a linear one.
	
	\begin{conjecture}
		The class of $\overline{h}$-perfect graphs is linearly $\chi$-bounded.
	\end{conjecture}

	\paragraph{Acknowledgements.} 
	We thank Dabeen Lee for contributing to discussions in the early stages of this project, Andr\'as Seb\H{o} for helpful feedback, and Sebastian Wiederrecht for sharing this problem with us at the \href{https://www.math.sinica.edu.tw/www/file_upload/conference/202402Rim/index.html}{2024 Pacific Rim Graph Theory Group Workshop}. Thanks also go to Bruce Reed for organising and Academia Sinica for hosting this workshop.
	Most of this research was conducted during visits to Princeton University and the IBS. 
	We are grateful to both institutions for their hospitality.

	{
		\fontsize{11pt}{12pt}
		\selectfont

		\bibliographystyle{Illingworth.bst}
		\bibliography{main_refs.bib}
	}
	
\end{document}